\documentclass[11pt,a4paper]{article}
\usepackage{pgf}
\usepackage{tikz}
\usepackage[symbol]{footmisc}
\usetikzlibrary{arrows,automata}
\usepackage{verbatim}
\usepackage{caption}
\usepackage{epsf,epsfig,amsfonts,amsgen,amsmath,amstext,amsbsy,amsopn,amsthm,amssymb}
\usepackage{color,graphicx,xcolor}
\usepackage{mathrsfs,hyperref}
\usepackage{enumerate}
\usepackage{url}
\usepackage{enumitem}
\usetikzlibrary{patterns}
\usepackage{algorithm, algorithmic}
\usepackage{longtable,multirow,array,booktabs}
\usepackage{changepage}

\newcounter{ordercounter}
\newcommand{\resetorder}{\setcounter{ordercounter}{0}}

\newcolumntype{L}{>{\centering\arraybackslash}p{1.2cm}}
\newcolumntype{C}{>{\centering\arraybackslash}p{1.0cm}}
\newcolumntype{H}{>{\raggedright\arraybackslash}p{9.5cm}} 
\newcolumntype{P}{>{\raggedright\arraybackslash}p{2.5cm}}
\makeatletter

\newenvironment{breakablealgorithm}[1][1cm]
{
\begin{center} %
\begin{adjustwidth}{#1}{#1}
\refstepcounter{algorithm}
\noindent\rule{\linewidth}{0.8pt}\par\nointerlineskip \kern2pt
\renewcommand{\caption}[2][\relax]{
{\raggedright\textbf{\ALG@name~\thealgorithm} ##2\par}

\ifx\relax##1\relax

\addcontentsline{loa}{algorithm}{\protect\numberline{\thealgorithm}##2}%

\else

\addcontentsline{loa}{algorithm}{\protect\numberline{\thealgorithm}##1}%

\fi

\kern2pt \par\nointerlineskip \noindent\rule{\linewidth}{0.8pt}\par\nointerlineskip \kern2pt

}

}{

\kern2pt \par\nointerlineskip \noindent\rule{\linewidth}{0.8pt}

\end{adjustwidth}
\end{center} %

}

\newcommand{\algorithmicbreak}{\textbf{break}}
\newcommand{\BREAK}{\STATE \algorithmicbreak}
\newcommand{\algorithmiccontinue}{\textbf{continue}}
\newcommand{\CONTINUE}{\STATE \algorithmiccontinue}
\newcommand{\Comment}[1]{\hfill\textcolor{gray}{// #1}}
\newcommand{\LineComment}[1]{%
  \item[]\hspace*{\algorithmicindent}\textcolor{gray}{// #1}%
}

\newcommand{\LongState}[1]{%
  \parbox[t]{\dimexpr\linewidth - \algorithmicindent - 1em\relax}{%
    \raggedright #1%
  }%
}
\makeatother

\usepackage{setspace}
\newtheorem{dfn}{Definition}[section]
\newtheorem{thm}[dfn]{Theorem}
\newtheorem{thmo}{Theorem}[section]

\newtheorem{clm}[dfn]{Claim}

\newtheorem{lem}[dfn]{Lemma}
\newtheorem{fact}[dfn]{Fact}

\newtheorem{problem}[thmo]{Problem}
\newtheorem{cons}[thmo]{Construction}

\def\QED{\hfill \rule{7pt}{7pt}}

\newcommand{\dH}{{\mathcal{H}}}
\newcommand{\dV}{{\mathcal{V}}}
\newcommand{\dE}{{\mathcal{E}}}
\newcommand{\dT}{{\mathcal{T}}}

\newcommand{\dC}{{\mathcal{C}}}

\newcommand{\dF}{{\mathcal{F}}}

\def\ex{\mathrm{ex}}
\def\sat{\mathrm{sat}}
\makeatletter

\newcommand{\Rmnum}[1]{\expandafter\@slowromancap\romannumeral #1@}
\makeatother

\definecolor{cream}{RGB}{253, 246, 227}

\title{Saturation numbers for $3$-uniform Berge-$K_4$}
\author{Yihan Chen\footnote{School of Mathematical Sciences, University of Science and Technology of China (\texttt{cyh2020@mail.ustc.edu.cn})}~~~~~~ Jialin He\footnote{(\texttt{majlhe@outlook.com})}~~~~~~ Tianying Xie\footnote{School of Mathematics and Statistics, Fuzhou University (\texttt{xiety@fzu.edu.cn})}~~~~~~}
\date{\today}

\begin{document}

\maketitle

\begin{abstract}
    The saturation number $\sat_r(n,\dF)$ is the minimum number of hyperedges in an $r$-uniform $\dF$-saturated hypergraph on $n$ vertices. 
    We determine this parameter for $3$-uniform Berge-$K_4$ hypergraphs, proving that $\sat_3(n,\text{Berge-}K_4)=n$ for $n =5,7,8$ and $n\ge 96$, while $\sat_3(6,\text{Berge-}K_4)=5$. 
    This resolves a problem posed by English, Kritschgau, Nahvi, and Sprangel~\cite{EKNS2024} for large $n.$    
    Using a computer search, we classify all extremal hypergraphs for $5\le n\le 8.$ For $n\geq 96$, we further show the existence of many non-isomorphic extremal families. 
    Our approach synthesizes structural insights with computational power.
\end{abstract}

\section{Introduction}
As a cornerstone of modern combinatorics, extremal graph theory is concerned with investigating the limits of structural properties in graphs and hypergraphs under various constraints. Given a family $\dF$ of (hyper)graphs, a (hyper)graph is called $\dF$-free if it contains no copy of any member of $\dF$. 
One of the central problems in extremal graph theory is the Tur\'{a}n problem: for a family $\dF$ of forbidden subgraphs, determine the maximum number $\ex(n,\dF)$  of (hyper)edges in an $n$-vertex $\dF$-free (hyper)graph. 
This problem is well-understood for complete graphs, starting with Mantel's theorem for $K_3$~\cite{Mantel1907} and Tur\'{a}n's theorem for $K_t$~\cite{Turan1941}, and asymptotically for general graphs by the Erd\H{o}s--Stone--Simonovits theorem~\cite{ES1966,ES1946}.

A (hyper)graph is \textit{$\dF$-saturated} if it is $\dF$-free but the addition of any new (hyper)edge creates a copy of some element of $\dF$. 
In contrast to Tur\'{a}n problem, the saturation problem asks for the minimum number of edges in an $\dF$-saturated (hyper)graph. 
In particular, the $r$-uniform saturation number of hypergraph $\dF$ is denoted by 
$$\sat_r(n,\dF)=\min \left\{ |E(\dH)|:|V(\dH)|=n, \dH\text{ is }r\text{-uniform and }\dF\text{-saturated}\right\}.$$ 
Saturation numbers and their variants have wide applications in combinatorics and are also relevant to other areas, such as bootstrap percolation~\cite{BBMR2012} and topology~\cite{CCKK2024,TT2025}.

The study of saturation numbers began with graphs was initiated independently by Zykov in 1949~\cite{Zykov1949} and by Erd\H{o}s, Hajnal, and Moon in 1964~\cite{EHM1964}, determining $\sat_2(n,K_p)=\binom{n}{2}-\binom{n-p+2}{2}$ and characterizing the unique extremal graph. 
Later, K\'aszonyi and Tuza~\cite{KT1986} established the linear upper bound $\sat_2(n,\mathscr{F}) = O(n)$ for any graph family $\mathscr{F}$ and solved the problem for several specific families (e.g., stars, paths, matchings). 
For further graph saturation results, we refer readers to \cite{BFP2010, 
CFFS2012, 
Tuza1989} and subsequent works.

For hypergraphs, Bollob\'{a}s~\cite{Bollobas1965} generalized the complete graph result to the $r$-uniform complete hypergraph $K_p^r$ in 1965, proving an exact formula $\sat_r(n,K_p^r)=\binom{n}{r}-\binom{n-p+r}{r}$ and showing the uniqueness of extremal hypergraph. 
Pikhurko~\cite{Pikhurko1999} later showed the general upper bound $\sat_r(n,\mathscr{F}) = O(n^{r-1})$ for any finite family $\mathscr{F}$ of $r$-uniform hypergraphs. Additional exact results for specific hypergraph families appear in~\cite{CS2024, CFFS2012, EKZ2023, EFT1991}.

Determining exact saturation numbers remains challenging even for graphs, and is considerably harder in the hypergraph setting, where only a few precise values are known. 
A particularly active area of research within this field concerns \textit{Berge hypergraphs}, a notion introduced by Gerbner and Palmer~\cite{GP2017} that has attracted considerable interest.
Given a graph $F$, a hypergraph $\dH$ is called a \textbf{Berge-$F$} if there exists a \textbf{bijection} $\phi: E(F) \to E(\dH)$ such that $\{u,v\} \subseteq \phi(uv)$ for every edge $uv \in E(F)$. 
The saturation number for Berge-$F$ was first systematically studied by English, Gordon, Graber, Methuku, and Sullivan~\cite{EGGMS2019}. 
They established several bounds and exact results when $F$ is a path, cycle, or matching, and conjectured that $\sat_r(n,\text{Berge-}F)=O(n)$ for any fixed graph $F$.
Recently, English, Kritschgau, Nahvi, and Sprangel~\cite{EKNS2024} determined the asymptotic saturation number for Berge cliques, showing that
$$\sat_r(n,\text{Berge-}K_p) \sim \frac{p-2}{r-1}\,n.$$
Despite this asymptotic progress, exact values are known only for the smallest case $p=3$. This gap naturally leads to the problem of determining the exact saturation number $\sat_r(n,\text{Berge-}K_p)$ for $p \ge 4$. 
Based on an extremal construction, English et al. posed the following concrete problem.

\begin{problem}[\cite{EKNS2024}]
    Determine $\sat_3(n,\text{Berge-}K_4)$ exactly. In particular, is $\sat_3(n,\text{Berge-}K_4)=n$ when $n$ is odd?
\end{problem}

In this paper, we resolve the problem for all integers $n$ with $5 \le n \le 8$ and $n \ge 96$. Using parallel computation (see~\ref{Sec:Appendix A}), we determine all extremal hypergraphs for $n = 5,6,7,8$, and for $n \ge 96$ we exhibit many non-isomorphic families of extremal hypergraphs. Our main result is the following.

\begin{thm}\label{mainthm}
    For $n = 5,7,8$ and $n\geq 96$, we have $$\sat_3(n,\mathrm{Berge}\text{-}K_4)=n,$$ 
    while $\sat_3(6,\mathrm{Berge}\text{-}K_4)=5$. 
    Moreover, there exist many non-isomorphic extremal hypergraphs.
\end{thm}

We remark that the constant $96$ is not optimal; a more careful analysis could lower it, though the proof would become more tedious. Nevertheless, the remaining gap cannot be easily closed with current methods.

\subsection{Definitions and notations}
Let $\dH = (V(\dH),E(\dH))$ be a hypergraph, we denote by $v(\dH)$ and $e(\dH)$ the number of vertices and hyperedges of $\dH$, respectively. The minimum degree of $\dH$ is denoted by $\delta(\dH)$. 
A hypergraph is called an $r$-uniform hypergraph (or $r$-graph for short) if all its hyperedges have cardinality $r$. Assume that $\dH$ is a $3$-graph, and let $v\in V(\dH)$ be a vertex, we set $N^d_{\dH}(v)=\{(v_i,v_j): vv_iv_j\in E(\dH)\}$, $N_{\dH}(v)=\left\{v_i: \{v,v_i\}\subset e \mbox{ for some } e\in E(\dH)\right\}$, and $d_{\dH}(v)=|\{e\in E(\dH):v\in e\}|$. For a pair of vertices $(u,v)\in V(H)\times V(H)$, we set $N_{\dH}(u,v)=\{w\in V(\dH):uvw\in E(\dH)\}$ and $d_{\dH}(u,v)=|N_{\dH}(u,v)|$.

Recall that $\dH$ is a Berge-$F$ if there exists a bijection $\phi: E(F) \to E(\dH)$ such that $\{u,v\}\subseteq \phi(uv)$ for every edge $uv\in E(F)$. The \textbf{core} of $\dH$ is the set of vertices that correspond to the vertices of $F$.
For $u, v\in V(\dH)$, we say that the pair $(u,v)$ is $\mathbf{good}$ if $\dH+uv$ creates a new Berge-$K_4$; otherwise, the pair is $\mathbf{bad}$. 
Consequently, if $(u,v)$ is good, then the core of the newly formed Berge-$K_4$ must contain both $u$ and $v$. 

For an integer $\ell \ge r$, the \textit{$r$-uniform tight cycle} of length $\ell$, denoted $\dC_\ell^r$, is the $r$-graph with vertex set $V(\dC^r_\ell)=\{v_1,v_2,\cdots,v_\ell\}$ and hyperedge set $E(\dC^r_\ell)=\left\{v_iv_{i+1}\cdots v_{i+r-1}: i\in [\ell]\right\}$, where indices are taken modulo $\ell$.

\subsection{Organization of the paper}
The paper is structured as follows. In Section~\ref{Sec: Upper bound}, we construct hypergraphs to establish the upper bound $\sat_3(n,\text{Berge-}K_4)\le n$. Section~\ref{Sec: Lower bound} completes the proof of Theorem~\ref{mainthm} by providing a matching lower bound for $n\ge 96$ and discussing the extremal hypergraphs. The algorithm we developed for searching for uniform Berge-$K_\ell$-saturated hypergraphs is detailed in Section~\ref{Sec: Algorithm}. We conclude with a discussion of further problems in Section~\ref{Sec: Conclusion}. Computational results are presented in~\ref{Sec:Appendix A}.

\section{Upper bound}
\label{Sec: Upper bound}
In this section, we establish $\sat_3(n,\text{Berge-}K_4)\le n$ for every $n\geq 5$ by constructing two families of $3$-uniform Berge-$K_4$-saturated hypergraphs with $n$ vertices and $n$ hyperedges, one for odd $n$ and one for even $n$. Note that the case $n=6$ is special, we discuss it in \ref{Sec:Appendix A}. The following discussion concerns the cases where $n\neq 6$.

We begin by introducing the tool used to generate large Berge-$K_4$-saturated hypergraphs from smaller ones. Let $\dT$ be the hypergraph on $\{a_1,a_2,x_1,x_2\}$ with hyperedge set $\{a_1a_2x_1,a_1a_2x_2\}$. For a $3$-uniform Berge-$K_4$-saturated hypergraph $\dH=(\dV, \dE)$ and a vertex pair $(u,v)\in \dV \times \dV$, we say that $\dT$ can be added on $(u,v)$ if, by identifying $(x_1,x_2)$ with $(u,v)$ and attaching $\dT$ to $\dH$, the resulting hypergraph remains Berge-$K_4$-saturated. We have the following lemma.
\begin{lem}\label{Lemma: Add T}
Let $\dH=(\dV, \dE)$ be a $3$-uniform Berge-$K_4$-saturated hypergraph, and let $(u,v)$ be a pair of vertices in $\dV$. If $\dT$ can be added on $(u,v)$, then for any positive integer $k$, $k$ copies of $\dT$ can be added on $(u,v)$ simultaneously. Moreover, it is impossible to add two copies of $\dT$ simultaneously on two distinct pair of vertices in $\dV$.
\end{lem}
\begin{proof}
    We prove the first statement by induction on $k$. It holds trivially for $k = 1$. Assume that the statement holds for $k = t$. Let $k = t+1\ge 2$, we prove that after adding $k$ copies of $\dT$ on $(u,v)$ the resulting hypergraph remains Berge-$K_4$-saturated.

    We denote the resulting hypergraph by $\dH_k$, it is easy to see that $\dH_k$ is Berge-$K_4$-free. Because every edge of a Berge-$K_4$ contains two vertices of degree at least $3$, which implies that the edges of $\dT$'s can not appear in any Berge-$K_4$. Together with $\dH$ is Berge-$K_4$-free, the claim follows.

    We now prove that adding any non-hyperedge $e$ of $\dH_k$ yields a Berge-$K_4$. For a given $\dT$, let $\dT_A:= \{v\in V(\dT): deg_{\dT}(v)=2\}$. We note that for every copy of $\dT_A$ in $\dH_k$, $\dH_k\left[V(\dH_k)\setminus \dT_A\right]$ is isomorphic to $\dH_{k-1}$.  Since $\dH$ is a Berge-$K_4$-saturated hypergraph, we only need to consider the non-hyperedges which intersects some $\dT_A$. If $e$ intersects at most one copy of $\dT_A$, we delete a copy of $\dT_A$ that is disjoint from $e$. By the induction hypothesis, we obtain a Berge-$K_4$. If $e$ intersects at least two copies of $\dT_A$, without loss of generality, assume that the two copies are $\dT_A^1$ and $\dT_A^2$ and let $a_1\in e\cap \dT_A^1, a_2\in e\cap \dT_A^2$. We have the following claim.
    \begin{clm}\label{ref:Claim u v adjacent}
        $u$ and $v$ are adjacent in $\dH$.
    \end{clm}
    If the claim is true, we obtain a Berge-$K_4$ with core $\{a_1,a_2,u,v\}$ immediately.
    \begin{proof}[Proof of Claim~\ref{ref:Claim u v adjacent}]
        We delete some copy of $\dT_A$ in $\dH_{k}$ to obtain $\dH_{k-1}$. Let $\{a_3,a_4\}$ be a copy of $\dT_A$ in $\dH_{k-1}$, and let $x\in V(\dH_{k-1})$ be a vertex such that $x\notin \{u,v,a_3,a_4\}$. Then $\{a_3,a_4,x\}$ forms a non-hyperedge in $\dH_{k-1}$. By the induction hypothesis, adding this non-hyperedge yields a Berge-$K_4$. Clearly, $a_3$ and $a_4$ can not appear in the core of the Berge-$K_4$ simultaneously, which implies that the core of the Berge-$K_4$ is either $\{u,v,x,a_3\}$ or $\{u,v,x,a_4\}$. In both cases, there should be a hyperedge in $\dH$ connects $u$ and $v$, which yields the claim.
    \end{proof}

    We now turn to the second statement, suppose that there are two copies of $\dT$, say $\dT^1$ and $\dT^2$, added simultaneously on two different vertex pairs. Let $\dT_A^1 = \{a_1,a_2\}$ and $\dT_A^2 = \{a_3,a_4\}$, then $\{a_1,a_2,a_3\}$ is not a hyperedge. We claim that adding it to the hypergraph does not create a new Berge-$K_4$, which is a contradiction to the definition of Berge-$K_4$-saturated. Indeed, since $a_1$ and $a_2$ can not be in the core of a Berge-$K_4$ simultaneously, and in the original hypergraph (before the addition), $a_3$ has at most one common neighbor with each of $a_1$ and $a_2$, which implies that $a_3$ and $a_1$ (or $a_3$ and $a_2$) can not be in the core of a Berge-$K_4$ simultaneously. Thus, all the pairs of vertices $(a_1,a_2)$, $(a_1,a_3)$, and $(a_2,a_3)$ are bad and we are done.
\end{proof}

Next, we provide the constructions for the cases when $n\ge 5$ is odd and when $n\ge 8$ is even  respectively; Lemma~\ref{Lemma: Add T} guarantees that both constructions are Berge-$K_4$-saturated with $n$ edges. 

\begin{figure}[H]
    \centering
  \begin{minipage}{0.48\textwidth}
    \centering
    \includegraphics[trim=300 190 220 200, 
    clip, width=0.7\textwidth]{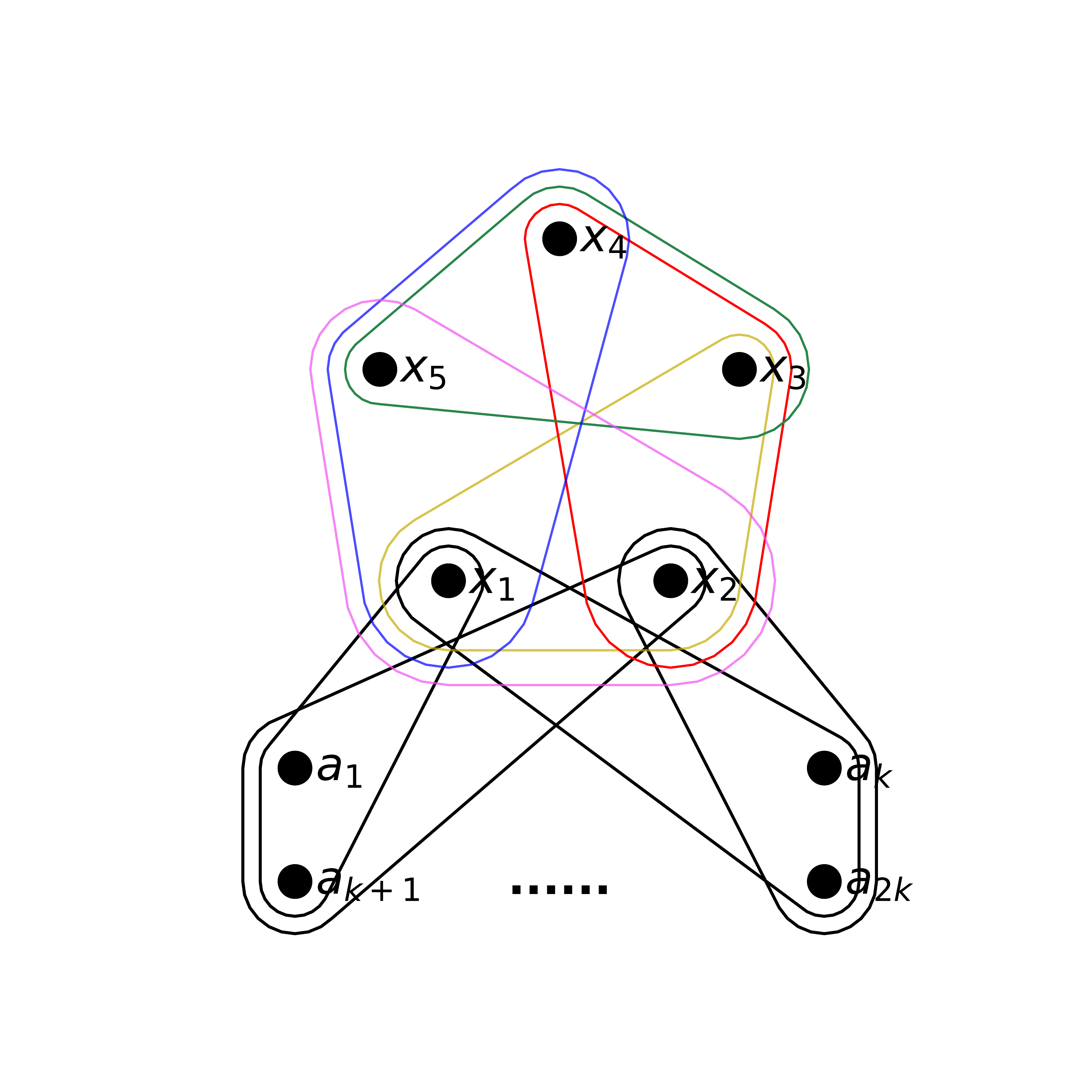}
    \caption{Figure of Construction~\ref{Construction: odd n}}
    \label{fig:1}
  \end{minipage}
  \hfill 
  \begin{minipage}{0.48\textwidth}
    \centering
    \includegraphics[trim=300 190 220 200, 
    clip, width=0.7\textwidth]{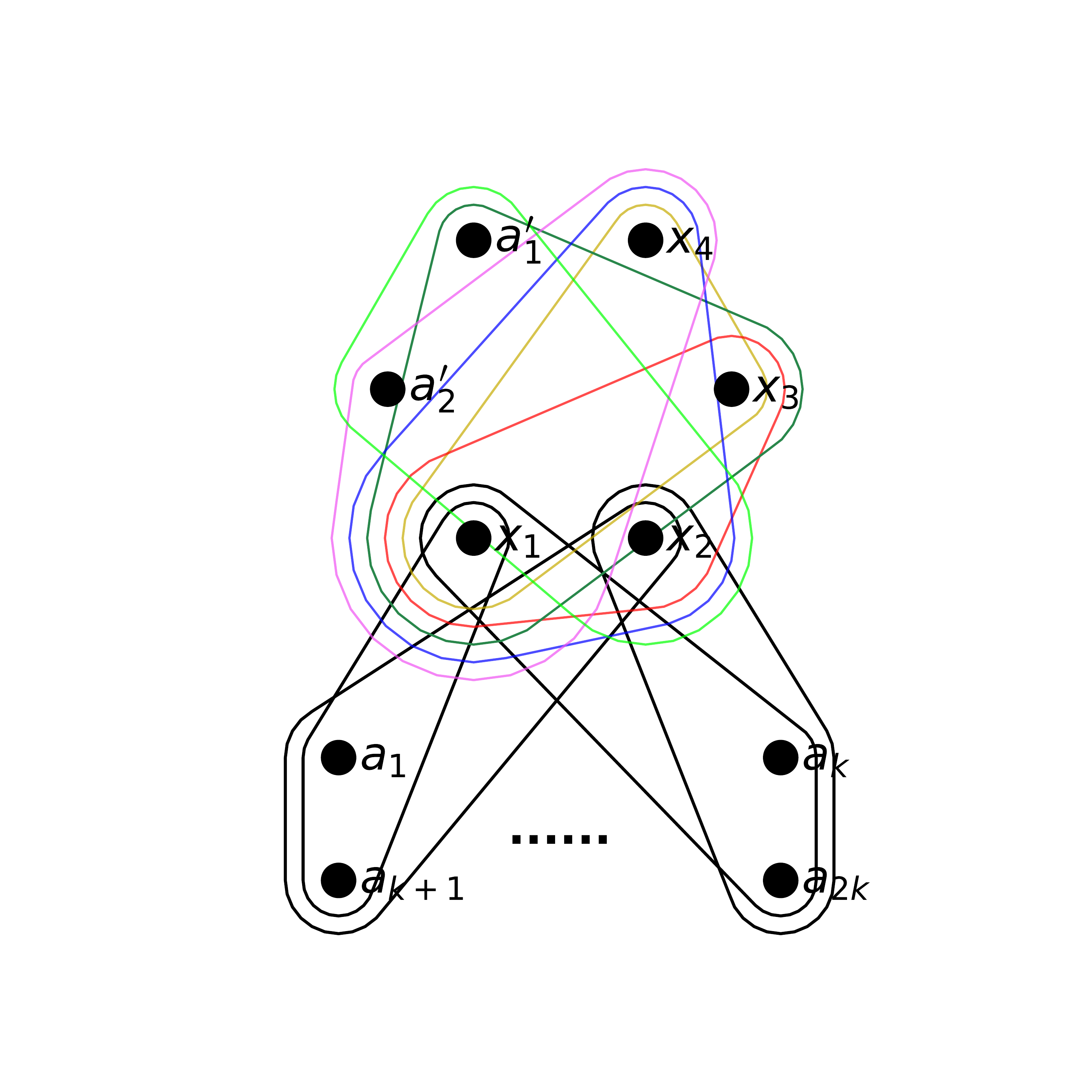}
    \caption{Figure of Construction~\ref{Construction: even n}}
    \label{fig:2}
  \end{minipage}
\end{figure}

\begin{cons}\label{Construction: odd n}
    Let $n\ge 5$ be an odd integer, set $k=\frac{n-5}{2}$, and let $\dH$ be the 3-uniform tight $5$-cycle, $\dC_5^3$, with vertex set $\{x_1,x_2,x_3,x_4,x_5\}$ and hyperedge set $\{x_1x_2x_3,x_2x_3x_4,x_3x_4x_5,x_4x_5x_1,x_5x_1x_2\}$ (isomorphic to the No.$6$ hypergraph of order $5$ in Table~\ref{Table: Extremal hypergraphs}) which serves as the base hypergraph. We add $k$ copies of $\dT$ on the vertex pair $(x_1,x_2)$. 
    See Figure 1 as above.
\end{cons}

\begin{cons}\label{Construction: even n}
    Let $n\ge 8$ be an even integer, set $k=\frac{n-6}{2}$, and let $\dH$ be the 3-graph with vertex set $\{x_1,x_2,x_3,x_4,a_1^\prime,a_2^\prime\}$ and hyperedge set $\{x_1x_2x_3,x_1x_3x_4,x_1x_3a_1^\prime,x_1x_2x_4,x_1x_4a_2^\prime, x_2a_1^\prime a_2^\prime\}$. Verified by computer, $\dH$ is Berge-$K_4$-saturated and admits the addition of $\dT$ on the vertex pair $(x_1,x_2)$. We take $\dH$ as the base hypergraph and add $k$ copies of $\dT$ on the vertex pair $(x_1,x_2)$.
    See Figure 2 as above.
\end{cons}

\section{Lower bound and extremal hypergraphs}\label{Sec: Lower bound}
This section is devoted to proving Theorem~\ref{mainthm}. In Subsection~\ref{Sec:lower bound n large}, we establish the lower bound $\sat_3(n, \text{Berge-}K_4)\ge n$ for all $n\ge 96$. Subsection~\ref{Sec: pf of main thm} analyzes the smaller cases ($n\le 8$) and discusses the structure of extremal hypergraphs. 

We begin with several useful lemmas and facts.
\subsection{Lemmas and facts}
We first state two elementary facts that will be used frequently.
\begin{fact}\label{Fact1}
    Let $n\geq 5$ and $\dH$ be a Berge-$K_4$-saturated 3-graph on $n$ vertices. The following hold. 
    \begin{enumerate}
		\item[(i).] If $u,v\in V(\dH)$ and $(u,v)$ is good, then any new Berge-$K_4$ in $\dH+uv$ must contain all hyperedges containing both $u$ and $v$. Moreover, $d_{\dH}(u,v)\leq 2$ and $d_{\dH}(w)\geq 3$ for every $w\in N_{\dH}(u,v)$. 
		\item[(ii).] For $v_1v_2v_3\in E(\dH)$ with $d_{\dH}(v_1)\leq d_{\dH}(v_2)\leq d_{\dH}(v_3)$. If $d_{\dH}(v_1)\leq 2$, then $(v_2,v_3)$ is bad. If $d_{\dH}(v_1)\leq d_{\dH}(v_2)\leq 2$, then $(v_1,v_2)$, $(v_1,v_3)$ and $(v_2,v_3)$ are bad. If $d_{\dH}(v_1)\leq d_{\dH}(v_2)\leq d_{\dH}(v_3)\leq 2$, then $(v_1,u)$, $(v_2,u)$ and $(v_3,u)$ are bad for any $u\in V(\dH)\setminus \{v_1,v_2,v_3\}$.
    \end{enumerate}
\end{fact}
\begin{proof}
For the first statement,
if $(u,v)$ is good, then any new Berge-$K_4$ must contain all hyperedges containing both $u$ and $v$, otherwise $H$ contains a Berge-$K_4$. Thus the core of any new Berge-$K_4$ must contain every vertex in $N_{\dH}(u,v)$, this implies that $d_{\dH}(u,v)\leq 2$ and $d_{\dH}(w)\geq 3$ for every $w\in N_{\dH}(u,v)$.

For the second statement,
let $v_1v_2v_3\in E(\dH)$ and $d_{\dH}(v_1)\leq d_{\dH}(v_2)\leq d_{\dH}(v_3)$. If $d_{\dH}(v_1)\leq 2$, then $(v_2,v_3)$ is bad by (i). 
Moreover, if $d_{\dH}(v_1)\leq d_{\dH}(v_2)\leq 2$, then there are at most three hyperedges that contain $v_1$ or $v_2$, which means that $v_1$ and $v_2$ can not be in the core of a Berge-$K_4$ simultaneously, thus $(v_1,v_2)$ is bad.
If $(v_1,v_3)$ is good, as $d_{\dH}(v_1)\leq 2,$ $v_1$, $v_2$, and $v_3$ must be in the core of the Berge-$K_4$ created by adding $v_1v_3$ in $\dH$, which contradicts that $v_1$ and $v_2$ can not be in the core of a Berge-$K_4$ simultaneously. Thus, $(v_1,v_3)$ is bad, and $(v_2,v_3)$ is bad can be proved similarly.  
Finally, let $d_{\dH}(v_1)\leq d_{\dH}(v_2)\leq d_{\dH}(v_3)\leq 2$. For every vertex $u\in V(\dH)\setminus \{v_1,v_2,v_3\}$ and every $v_i\in \{v_1,v_2,v_3\}$, suppose, for a contradiction, that $(u,v_i)$ is good, then the core of any new Berge-$K_4$ in $\dH+uv_i$ must contain two vertices in $\{v_1,v_2,v_3\}$, which is a contradiction to $d_{\dH}(v_1)\leq d_{\dH}(v_2)\leq d_{\dH}(v_3)\leq 2$. Thus $(v_1,u)$, $(v_2,u)$ and $(v_3,u)$ are bad and we are done.
\end{proof}

Next, we state two structural lemmas that will be used in later proofs. Let $\dH$ be a Berge-$K_4$-saturated $3$-graph on $n$ vertices with the minimum number of hyperedges. The first lemma shows that under certain degree and ``goodness'' conditions,  $\dH$ must contain a copy of $\dC^3_5$.
The second lemma gives a lower bound on the minimum degree of $\dH$.
\begin{lem}\label{lem1:tight C5}
	Let $n\geq 5$ be an integer, and let $\dH$ be a Berge-$K_4$-saturated $3$-graph on $n$ vertices with the minimum number of hyperedges. 
    Suppose there exists a vertex $v_1\in V(\dH)$ such that $d_{\dH}(v_1)=3$, every pair $(v_i,v_j)\in N^d_{\dH}(v_1)$ is good, and moreover, the pair $(v_1,v_k)$ is good for every $v_k\in N_{\dH}(v_1)$. Then there exist four vertices $v_2, v_3, v_4, v_5$ such that $\dH[\{v_1, v_2, v_3, v_4, v_5\}]$ contains a $\dC^3_5$.
\end{lem}

\begin{proof}
	By Constructions~\ref{Construction: odd n} and \ref{Construction: even n}, we have $e(\dH)\leq n$. 
	Let $v_1\in V(\dH)$ be the vertex such that $d_{\dH}(v_1)=3$, every pair $(v_i,v_j)\in N^d_{\dH}(v)$ is good, and $(v_1,v_k)$ is good for every $v_k\in N_{\dH}(v_1)$. Let $e_1,e_2,e_3\in E(\dH)$ be the three hyperedges containing $v_1$, without loss of generality, set $e_1=v_1v_2v_3$. Since $(v_2,v_3)$ is good by Fact~\ref{Fact1} (i), the core of any Berge-$K_4$ in $\dH+v_2v_3$ contains $v_1$. Assume that one of such Berge-$K_4$'s has core $\{v_1,v_2,v_3,v_4\}$, and $e_1,e_2,e_3,e_4,e_5$ are its hyperedges in $\dH$ such that $v_1, v_2\in e_2$, $v_1,v_4\in e_3$,  $v_2, v_4\in e_4$ and $v_3,v_4\in e_5$. It is easy to see that $|e_1\cap e_2| = |\{v_1,v_2\}|=2$.

    We first claim that $v_2\not\in e_3$. Indeed, if not, then we have $e_1\cap e_2\cap e_3 =\{v_1,v_2\}$. Since $(v_1,v_2)$ is good, Fact~\ref{Fact1} (i) implies that $d_{\dH}(v_1,v_2)\leq 2$, which contradicts $d_{\dH}(v_1,v_2)\ge 3$. Hence $|e_1\cap e_3|\le 2$. 
    
    In the following of the proof, we proceed by analyzing the sizes of the intersections $|e_1\cap e_3|$ and $|e_2\cap e_3|$. 
    If $|e_1\cap e_3|=|e_2\cap e_3|=1$, then there exist two vertices $v_5,v_6\in V(\dH)\setminus\{v_1,v_2,v_3,v_4\}$ such that $e_2=v_1v_2v_5$ and $e_3=v_1v_4v_6$. A contradiction follows easily, since $(v_4,v_6)$ is good and the three edges containing $v_1$ are $e_1$, $e_2$, and $e_3$.   
    
    If $|e_1\cap e_3|=|e_2\cap e_3|=2$, then we have $e_2=v_1v_2v_4$ and $e_3=v_1v_3v_4$. Since $(v_2,v_4)$ is good and $\{v_2,v_4\}= e_2\cap e_4$, the core of the new Berge-$K_4$ in $\dH+v_2v_4$ must contain $v_1$ and the vertex in $e_4\setminus \{v_2,v_4\}$. As $N_{\dH}(v_1)=\{v_2,v_3,v_4\}$, the core must be $\{v_1,v_2,v_3,v_4\}$, which implies $e_4=v_2v_3v_4$. Similarly, by $(v_3,v_4)$ is good, we have that $e_5=v_2v_3v_4=e_4$, which is a contradiction. 

    We only need to consider the cases where either $|e_1\cap e_3|=2$, $|e_2\cap e_3|=1$, or $|e_1\cap e_3|=1$, $|e_2\cap e_3|=2$. 
    Without loss of generality, we may assume that $|e_1\cap e_3|=2$, $|e_2\cap e_3|=1$ and there is vertex $v_5\in E(\dH)$ such that $e_2=v_1v_2v_5$ and $e_3=v_1v_3v_4$. 
    Let $E_1=E(\dH)\setminus\{e_1,e_2,e_3\}$. If there are three hyperedges in $E_1$ that form a Berge-$K_3$ with core in $\{\{v_2,v_3,v_4\},\{v_2,v_3,v_5\}, \{v_2,v_4,v_5\}, \{v_3,v_4,v_5\}\}$, then one can easily find a Berge-$K_4$ in $\dH$, which is a contradiction. Henceforth, we assume that no such triple of hyperedges in $E_1$ exists. Since $(v_1,v_2)$ is good and $v_1v_2v_3,v_1v_2v_5 \in E(\dH)$, there are two hyperedges $e_6,e_7\in E_1$ such that $v_2,v_5\in e_6$ and $v_3,v_5\in e_7$. 
    Since $(v_3,v_4)$ is good and $N_{\dH}(v_1)=\{v_2,v_3,v_4,v_5\}$, the core of the new Berge-$K_4$ should be $\{v_3,v_4,v_1,v_2\}$ or $\{v_3,v_4,v_1,v_5\}$, which implies that there are two hyperedges $e_8,e_9\in E_1$ such that $v_3,v_2\in e_8$ and $v_4,v_2\in e_9$ or $v_3,v_5\in e_8$ and $v_4,v_5\in e_9$. 
    If $v_3,v_2\in e_8$, we claim that $e_8\in \{e_4,e_5\}$ and $e_8\in \{e_6,e_7\}$, this implies that $e_8=v_2v_3v_4=v_2v_3v_5$, which is a contradiction. Indeed, if $e_8\notin \{e_4,e_5\},$ then $e_4,e_5,e_8\in E_1$ can form a Berge-$K_3$ with core $\{v_2,v_3,v_4\}$, and if $e_8\notin \{e_6,e_7\},$ then $e_6,e_7,e_8\in E_1$ can form a Berge-$K_3$ with core $\{v_2,v_3,v_5\}$, both will induce contradictions. 
    Thus we have $v_3,v_5\in e_8$ and $v_4,v_5\in e_9$. Since $E_1$ does not contains a Berge-$K_3$ with core in $\{\{v_2,v_4,v_5\}, \{v_3,v_4,v_5\}\}$, $e_4, e_6, e_9$ can not be three distinct hyperedges. Hence we have $v_2v_4v_5\in E_1$. We also have that $e_5, e_8, e_9$ can not be three distinct hyperedges, which implies $v_3v_4v_5\in E_1$. Hence $\dH[\{v_1, v_2, v_3, v_4, v_5\}]$ contains a $\dC^3_5$ with hyperedges $\{v_1v_2v_5, v_2v_4v_5, v_3v_4v_5, v_1v_3v_4, v_1v_2v_3\}$, which completes the proof of Lemma~\ref{lem1:tight C5}.
\end{proof}

\begin{lem}\label{lem2:delta(H) geq 2}
    Let $n\geq 7$ be an integer and $\dH$ be a Berge-$K_4$-saturated $3$-graph on $n$ vertices with the minimum number of hyperedges, then we have $\delta(\dH)\geq 2$.
\end{lem}

\begin{proof}
	By Constructions~\ref{Construction: odd n} and \ref{Construction: even n}, we have $e(\dH)\leq n$. 

    We first prove that $\delta(\dH)\geq 1$. Suppose for the contrary that there is a vertex $v_0\in V(\dH)$ with $d_{\dH}(v_0)=0$, then $v_0v_iv_j\notin E(\dH)$ implies that $(v_i,v_j)$ is good for every pair $(v_i,v_j)\in (V(\dH)\setminus \{v_0\})\times (V(\dH)\setminus \{v_0\})$. Let $\dH_0=\dH-\{v_0\}$. Since every pair $(v_i,v_j)\in V(\dH_0)\times V(\dH_0)$ is good, we have $\delta(\dH_0)\geq 2$. 
    Moreover, for any vertex $v_i\in V(\dH_0)$, let $v_iv_jv_k\in E(\dH)$ be a hyperedge incident to $v_i$. Since $(v_j,v_k)$ is good in $\dH$, the core of any Berge-$K_4$ in $\dH+v_jv_k$ must contain $v_i$, which implies $d_{\dH_0}(v_i)=d_{\dH}(v_i)\geq 3$, and hence $\delta(\dH_0)\geq 3$.
    We claim that $\delta(\dH_0)=3$. Indeed, if $\delta(\dH_0)\geq 4$, then we have $e(\dH)=e(\dH_0)\geq \left\lceil\frac{4(n-1)}{3}\right\rceil\geq n+1$ for $n\geq 5$, which contradicts that $e(\dH)\leq n$. 
    
    Let $v_1\in V(\dH_0)$ be the vertex with $d_{\dH_0}(v_1)=3$. Since any pair $(v_i,v_j)\in V(\dH_0)\times V(\dH_0)$ is good, Lemma~\ref{lem1:tight C5} implies that there are four vertices $v_2,v_3,v_4,v_5\in V(\dH)$ such that $v_1v_2v_3, v_2v_3v_4$, $v_3v_4v_5, v_4v_5v_1, v_1v_2v_5\in E(\dH)$. In other words, $v_1,v_2,v_3,v_4,v_5$ form a $\dC^3_5$.
    If $n=7$, then there is a vertex $v_6\in V(\dH_0)\setminus\{v_1,v_2,v_3,v_4,v_5\}$. Since $d_{\dH}(v_6)\geq\delta(\dH_0)\geq 3$, there are at least three hyperedges in $E(\dH)\setminus\{v_1v_2v_3, v_2v_3v_4, v_3v_4v_5, v_4v_5v_1, v_1v_2v_5\}$, thus $e(\dH)\geq 5+3=n+1$, a contradiction.
    If $n=8$, then there are two vertices $v_6,v_7\in V(\dH_0)\setminus\{v_1,v_2,v_3,v_4,v_5\}$. Since $(v_6,v_7)$ is good, by Fact~\ref{Fact1} (i), we have $d_{\dH}(v_6,v_7)\leq 2$. Thus $d_{\dH}(v_6), d_{\dH}(v_7)\geq\delta(\dH_0)\geq 3$ implies that there are at least four hyperedges that contain $v_6$ or $v_7$. Hence $e(\dH)\geq 5+4=n+1$, a contradiction.
    If $n=9$, then there are three vertices $v_6,v_7,v_8\in V(\dH_0)\setminus\{v_1,v_2,v_3,v_4,v_5\}$. Since $(v_6,v_7)$, $(v_7,v_8)$ and $(v_6,v_8)$ are good, by Fact~\ref{Fact1} (i), we have $d_{\dH}(v_6,v_7), d_{\dH}(v_7,v_8), d_{\dH}(v_6,v_8)\leq 2$. Moreover, if $d_{\dH}(v_6,v_7)=1$, then there are at least five hyperedges that contain $v_6$ or $v_7$, which implies that $e(\dH)\geq 5+5=10=n+1$, a contradiction. Thus $d_{\dH}(v_6,v_7)=d_{\dH}(v_7,v_8)=d_{\dH}(v_6,v_8)=2$ and there are at most four hyperedges which contain at least one vertex in $\{v_6,v_7,v_8\}$ and moreover, we have $v_6v_7v_8\in E(\dH)$. Otherwise there are at least six distinct hyperedges which contain exact two vertices in $\{v_6,v_7,v_8\}$, a contradiction. Since $(v_6,v_7)$ is good and $v_6v_7v_8\in E(\dH)$, the core of Berge-$K_4$ must contain $v_6$, $v_7$ and $v_8$, which means there are at least five hyperedges which contain at least one vertex in $\{v_6,v_7,v_8\}$, a contradiction.
    If $n\geq 10$, by the previous analysis we know that $(v_1,v_i)$ is good for any $v_i\in V(\dH_0)\setminus\{v_1,v_2,v_3,v_4,v_5\}$. Since $d_{\dH}(v_1)=3$ and $v_1v_2v_3, v_4v_5v_1, v_1v_2v_5\in E(\dH)$, the core of any new Berge-$K_4$ in $\dH+v_1v_i$ must be $\{v_1, v_i, v_{j_1}, v_{j_2}\}$ for some $v_{j_1}, v_{j_2}\in \{v_2,v_3,v_4,v_5\}$. Thus there are two hyperedges $e_1, e_2\in E(\dH)$ such that $v_i, v_{j_1}\in e_1$ and  $v_i, v_{j_2}\in e_2$, which means that there are at least
    $|V(\dH\setminus\{v_0,v_1,v_2,v_3,v_4,v_5\})|=n-6$
    distinct hyperedges outside the $\dC^3_5$ with at least one vertex in $\{v_2,v_3,v_4,v_5\}$.
    In particular, we have $\sum_{j=2}^{5}d_{\dH}(v_j)\geq \sum_{j=2}^{5}d_{\dC^3_5}(v_j)+n-6\geq 3|\{v_2,v_3,v_4,v_5\}|+n-6$, which implies that $\sum_{v\in V(H_0)}d_{\dH}(v)\geq 3|V(\dH_0)|+n-6$. Hence $e(\dH)\geq\left\lceil\frac{3(n-1)+n-6}{3}\right\rceil\geq n+1$, where the last inequality holds when $n\ge 10$,  which a contradiction to $e(\dH)\le n$. 
\medskip

    Next, we prove $\delta(\dH)\geq 2$. Suppose for the contrary that there is a vertex $u_1\in V(\dH)$ with $d_{\dH}(u_1)=1$, let $u_1u_2u_3\in E(\dH)$ be the hyperedge incident to $u_1$. Clearly, $(u_1,u_i)$ is bad for every $u_i\in V(\dH)\setminus\{u_1\}$. Thus for any pair $(u_i,u_j)\neq(u_2,u_3)$ with $u_i,u_j\in V(\dH)\setminus\{u_1\}$, $u_1u_iu_j\notin E(\dH)$ implies that $(u_i,u_j)$ is good, and hence $d_{\dH}(u_i)\geq 2$ for every $u_i\in V(\dH)\setminus\{u_1\}$. 
    Moreover, for every $u_i\in V(\dH)\setminus\{u_1\}$, there is at least one pair $(u_j,u_k)\neq(u_2,u_3)$ such that $u_iu_ju_k\in E(\dH)$. Since $(u_j,u_k)$ is good, it follows that $d_{\dH}(u_i)\geq 3$. 
    Furthermore, there are at most two vertices have degree more than $3$, otherwise, if there are three vertices having degree at least $4$, then $e(\dH)\geq \left\lceil\frac{1+3(n-1-3)+3\times 4}{3}\right\rceil=\left\lceil\frac{3n+1}{3}\right\rceil=n+1$, which contradicts that $e(\dH)\leq n$.
    Similarly, we have that for any two vertices $u_i,u_j\in V(\dH)\setminus\{u_1\}$, $d_{\dH}(u_i)+d_{\dH}(u_j)\leq 8$. Otherwise, $e(\dH)\geq \left\lceil\frac{1+3(n-1-2)+9}{3}\right\rceil=n+1$, which also contradicts our assumption that $e(\dH)\leq n$.
    
    \begin{clm}\label{no C_5^3}
        There is no 3-uniform tight 5-cycle in $\dH$.
    \end{clm}
    \begin{proof}
    Suppose that there is a 3-uniform tight 5-cycle $\dC$ with $V(\dC)=\{w_1,w_2,w_3,w_4,w_5\}$ and $E(\dC)=\{w_1w_2w_3, w_2w_3w_4, w_3w_4w_5, w_4w_5w_1, w_5w_1w_2\}$ in $\dH$. Since there are at most two vertices have degree more than 3, we may assume that $d_{\dH}(w_2)=d_{\dH}(w_3)=d_{\dH}(w_4)=3$.
    
    If $u_2,u_3\in \{w_1,w_2,w_3,w_4,w_5\}$, then $\dC+u_1u_2u_3$ contains a Berge-$K_4$, a contradiction to $\dH$ is Berge-$K_4$-free. 
    Thus $|V(\dC)\cap \{u_2,u_3\}|\leq 1$. 
    
    If $V(\dC)\cap\{u_2,u_3\}=\emptyset$, the previous analysis tells us that $(w_2,u_2)$ and $(w_2, u_3)$ are good. Since $N_{\dH}(w_2)=\{w_1,w_3,w_4,w_5\}$, the core of any Berge-$K_4$ in $\dH+w_2u_2$ must composed of $w_2,u_2$ and two vertices from $\{w_1,w_3,w_4,w_5\}$. Assume that a core is $\{w_2,u_2,w_{i_1},w_{i_2}\}$ for some $w_{i_1},w_{i_2}\in \{w_1,w_3,w_4,w_5\}$. Then there are two hyperedges $e_3,e_4$ such that $w_{i_1},u_2\in e_3$ and $w_{i_2},u_2\in e_4$, thus $d_{\dH}(w_{i_1}),d_{\dH}(w_{i_2})\geq 4$. It implies that $(w_{i_1},w_{i_2})=(w_1,w_5)$. 
    On the other hand, since $(w_2, u_3)$ is also good, there are two hyperedges $e_5,e_6$ such that $u_3,w_1\in e_5$ and $u_3,w_5\in e_6$. Then we claim that $e_3=e_5=w_1u_2u_3$, $e_4=e_6=w_5u_2u_3$, and $d_{\dH}(w_1)=d_{\dH}(w_5)=4$, and any vertex in $V(\dH)\setminus\{u_1, w_1, w_5\}$ has degree $3$. Otherwise, one can easily obtain that $3e(\dH)\ge 3(n-1)+1+3=3n+1$, which implies that $e(\dH)\geq n+1$, a contradiction to $e(\dH)\leq n$.
    Furthermore, since $(w_1,u_2)$ is good and $N_{\dH}(u_2)=\{u_1,u_3,w_1,w_5\}$, the core of any new Berge-$K_4$ in $\dH+w_1u_2$ must be $\{w_1,u_2,u_3,w_5\}$. But $w_5u_2u_3$ can not be mapped to both $w_5u_2$ and $w_5u_3$, a contradiction. 
    
    Therefore, $|V(\dC)\cap\{u_2,u_3\}|=1$. We may assume that $u_2\in V(\dC)$, and clearly we have $d_{\dH}(u_2)\geq 4$. Moreover, $d_{\dH}(w_2)=d_{\dH}(w_3)=d_{\dH}(w_4)=3$ implies that none of $w_2$, $w_3$, or $w_4$ can be adjacent to $u_3$. Without loss of generality, we assume $w_1=u_2$. Since $(w_2,u_3)$ is good, the core of the Berge-$K_4$ in $\dH+w_2u_3$ must contains $w_1=u_2$ and $w_5$. Hence $d_{\dH}(w_5)=d_{\dH}(u_2)=4$ and there is a vertex $x\in V(\dH)$ such that $u_3xw_5\in E(\dH)$. On the other hand, since $(w_5,u_3)$ is good and $N_{\dH}(w_5)=\{u_2,w_2,w_3,w_4,u_3,x\}$, the core of any Berge-$K_4$ in $\dH+w_5u_3$ must be $\{u_3,w_5,x,u_2\}$. By considering $N_{\dH}(u_2)=\{u_1,u_3,w_2,w_3,w_4,w_5\}$, we have $x\in \{w_2,w_3,w_4\}$. Together with $u_3xw_5\in E(\dH)$, it follows that $d_{\dH}(x)=4$, which contradicts $d_{\dH}(w_2)=d_{\dH}(w_3)=d_{\dH}(w_4)=3$. Thus, there is no $\dC_5^3$ in $\dH$.
    \end{proof}
    
    Recall that $(u_i,u_j)$ is good for any pair $(u_i,u_j)\neq(u_2,u_3)$ with $u_i,u_j\in V(\dH)\setminus\{u_1\}$, thus if there is a vertex $w\in V(\dH)\setminus\{u_1,u_2,u_3\}$ such that $d_{\dH}(w)=3$ and $wu_2u_3\not\in E(\dH)$, then every pair $(v_i,v_j)\in N^d_{\dH}(w)$ is good, and $(w,v_k)$ is good for every $v_k\in N_{\dH}(w)$. Hence Lemma~\ref{lem1:tight C5} implies that there is a $\dC^3_5$ in $\dH$. In the following, we prove that there is a such vertex in $V(\dH)$, which is a contradiction to Claim~\ref{no C_5^3}.
    As shown before, we have $d_{\dH}(u_i)\ge 3$ for $u_i\in V(\dH)\setminus\{u_1\}$, and the number of vertices with degree more than $3$ is at most $2$. Moreover, we have $d_{\dH}(u_2)+d_{\dH}(u_3)\leq 8$, which implies $d_{\dH}(u_2,u_3)\leq 4$.
    If $d_{\dH}(u_2,u_3)=4$, then by $n\geq 7$, there is a vertex $u_7$ with $d_{\dH}(u_7)=3$ and $u_7u_2u_3\not\in E(\dH)$, a contradiction.
    If $d_{\dH}(u_2,u_3)=3$, then for $n\geq 8$, there is a vertex $u_8$ such that $d_{\dH}(u_8)=3$ and $u_8u_2u_3\not\in E(\dH)$, a contradiction. For $n=7$, suppose that $V(\dH)=\{u_1,...,u_7\}$ and assume that $u_1u_2u_3, u_2u_3u_4, u_2u_3u_5\in E(\dH)$, then one can easily obtain a contradiction in the case where $d_{\dH}(u_6)=3$ or $d_{\dH}(u_7)=3$. Thus we assume $d_{\dH}(u_6)=d_{\dH}(u_7)=4$, it follows that $d_{\dH}(u_4)=d_{\dH}(u_5)=3$. In this case, it is easy to see that $N_{\dH}(u_6)\subseteq\{u_4,u_5,u_7\}$, which implies that $d_{\dH}(u_6)\le 3$, a contradiction.
    If $d_{\dH}(u_2,u_3)\le2$, then by $n\geq 7$, there is a vertex $u_7$ such that $d_{\dH}(u_7)=3$ and $u_7u_2u_3\not\in E(\dH)$, a contradiction.
    This completes the proof of Lemma~\ref{lem2:delta(H) geq 2}.
\end{proof}

\subsection{Proof of the lower bound for \texorpdfstring{$n\ge 96$}{}}\label{Sec:lower bound n large}
Equipped with the tools above, we now prove the lower bound. Let $n\geq 96$ be an integer and $\dH$ be an $n$-vertex $3$-uniform Berge-$K_4$-saturated hypergraph with the minimum number of edges.

We define the following vertex partition: $X=\{v\in V(\mathcal{H}): d_{\mathcal{H}}(v)\geq 3\}$, $A=\{v\in V(\mathcal{H})\setminus X: \mbox{every edge containing }v \mbox{ intersects X}\}$ and $B=V(\mathcal{H})\setminus (X\cup A)$. 
By Lemma~\ref{lem2:delta(H) geq 2}, $\delta(\dH)\geq 2$, thus every vertex in $A\cup B$ has degree exactly $2$. Moreover, by the definition of $A$, there is no hyperedge in $\dH[A]$.
The proof proceeds in two steps. 
First we prove that $B$ is empty.
Then we perform a case analysis according to the types of hyperedges between $X$ and $A$ to get the desired lower bound of $e(\dH)$.

\begin{lem}\label{cla:B is empty}
      $B=\emptyset$.
\end{lem}

\begin{proof}
    Suppose for the contrary that there is a vertex $v\in B$, then there are two vertices $v_1, v_2\in V(\dH)\setminus X$ such that $vv_1v_2\in E(\dH)$. Since $d_{\dH}(v)=d_{\dH}(v_1)=d_{\dH}(v_2)=2$, by Fact~\ref{Fact1} (ii), we have $(v,v_1)$, $(v,v_2)$, $(v_1,v_2)$, $(v,u)$, $(v_1,u)$ and $(v_2,u)$ are bad for every $u\in V(\dH)\setminus \{v,v_1,v_2\}$. Hence for every $u\in V(\dH)\setminus \{v,v_1,v_2\}$, $\dH$ is Berge-$K_4$-saturated implies that $vv_1u, vv_2u\in E(\dH)$, and thus $d_{\dH}(v)>2$, a contradiction. 
\end{proof}

Note that Lemma~\ref{cla:B is empty} implies that $|A|+|X|=n$. If $|A|\leq 2$, then $$e(\dH)\geq \left\lceil\frac{3|X|+2|A|}{3}\right\rceil=\left\lceil\frac{3n-|A|}{3}\right\rceil\geq n,$$ and we are done. So we may assume that $|A|\geq 3$.

All hyperedges between $A$ and $X$ can be divided into two types: 
\textbf{Type~1} consists of those containing exactly two vertices from $A$ and one vertex from $X$;
\textbf{Type~2} consists of those containing exactly one vertex from $A$ and two vertices from $X$.

Consider a vertex $a_1 \in A$ such that the two hyperedges containing $a_1$, say $e_1$ and $e_2$, are both of Type~1. 
Then $(e_1,e_2)$ must fall into one of the following three patterns:
$$
    (a)\ e_1=a_1a_2x_1\ \text{and}\ e_2=a_1a_2x_2;
    (b)\ e_1=a_1a_2x_1\ \text{and}\ e_2=a_1a_3x_1;
    (c)\ e_1=a_1a_2x_1\ \text{and}\ e_2=a_1a_3x_2,
$$
where $a_2,a_3\in A$ and $x_1,x_2\in X$.

The next lemma shows that whenever a vertex $a_1\in A$ has both of its incident hyperedges of Type~1, we may restrict our attention to Case~(c).

\begin{lem}\label{Lem:no case a,b}
    Suppose there is a vertex $a_1 \in A$ such that both of its incident hyperedges are of Type~1. 
    If Case~(a) occurs, then $e(\dH)\le n.$ Moreover, Case~(b) can never occur.
\end{lem}
\begin{proof}
For the first statement, if Case~(a) occurs, we may assume there are vertices $a_2\in A$, $x_1,x_2\in X$ such that $a_1a_2x_1, a_1a_2x_2\in E(\dH)$.

By Fact~\ref{Fact1} (ii), since $a_1a_2x_1, a_1a_2x_2\in E(\dH)$ and $d_{\dH}(a_1)=d_{\dH}(a_2)=2$, we have $(a_1, a_2)$, $(a_1,x_1)$, and $(a_1,x_2)$ are bad. Since $a_1a_2 v\not\in E(\dH)$ for any $v\in V(\dH)\setminus \{a_1, a_2, x_1, x_2\}$, $(a_1, v)$ or $(a_2, v)$ is good. If $(a_1, v)$ (resp. $(a_2, v)$) is good, then the core of any new Berge-$K_4$ in $\dH+a_1v$ (resp. $\dH+a_2v$) must be $\{a_1,v,x_1,x_2\}$ (resp. $\{a_2,v,x_1,x_2\}$). Thus there are three hyperedges $e_v^1, e_v^2, e_v^3\in E(\dH)$ such that $x_1, v\in e_v^1$, $x_2, v\in e_v^2$ and $x_1, x_2\in e_v^3$. Since each $e_v^1$ and $e_v^2$ can contain at most $2$ vertices outside $\{a_1, a_2, x_1, x_2\}$, we have 
$|\cup_{v\in V(\dH)\setminus \{a_1, a_2, x_1, x_2\}}\{e_v^1\}|,|\cup_{v\in V(\dH)\setminus \{a_1, a_2, x_1, x_2\}}\{e_v^2\}|\geq \left\lceil\frac{n-4}{2}\right\rceil$.
Moreover, since $\{a_1, a_2\}\cap e_v^3=\emptyset$, there is a vertex $v^\prime\in V(\dH)\setminus \{a_1, a_2, x_1, x_2\}$ such that $e_v^3=v^\prime x_1x_2$. Similarly, there is a vertex $v^{\prime\prime}\in V(\dH)\setminus \{a_1, a_2, x_1, x_2\}$ such that $e_{v^\prime}^3=v^{\prime\prime} x_1x_2$. 

If $v^{\prime\prime}\neq v^\prime$, then there are at least two hyperedges from $\cup_{v\in V(\dH)\setminus \{a_1, a_2, x_1, x_2\}}\{e_v^1, e_v^2, e_v^3\}$ containing both $x_1$ and $x_2$. 
If $v^{\prime\prime}=v^\prime$, then $e_{v^\prime}^3\not\in \cup_{v\in V(\dH)\setminus \{a_1, a_2, x_1, x_2\}}\{e_v^1, e_v^2\}$, that is, there is at least one hyperedge in $E(\dH)\setminus(\cup_{v\in V(\dH)\setminus \{a_1, a_2, x_1, x_2\}}\{e_v^1, e_v^2\})$ containing both $x_1$ and $x_2$. 
In both cases, we have $$d_{\dH}(x_1)+d_{\dH}(x_2)\geq 2\left\lceil\frac{n-4}{2}\right\rceil+2+|\{a_1a_2x_1, a_1a_2x_2\}|\geq n-4+2+2=n.$$ 
Combining with $3e(\dH)= \left(\sum_{x\in X\setminus\{x_1,x_2\}}d_{\dH}(x)\right)+d_{\dH}(x_1)+d_{\dH}(x_2)+2|A|$, this implies that
\begin{align*}
    e(\dH)\geq \left\lceil\frac{3(|X|-2)+n+2|A|}{3}\right\rceil
    = \left\lceil\frac{3n+|X|-6}{3}\right\rceil.
\end{align*}
Since $a_1x_1x_2\not\in E(\dH)$, $(a_1,x_1)$ and $(a_1,x_2)$ are bad implies that $(x_1,x_2)$ is good. And the core of every new Berge-$K_4$ in $\dH+x_1x_2$ must contains 4 vertices in $X$, thus $|X|\geq 4$ and therefore $e(\dH)\geq n$.

For the second statement, if Case~(b) occurs, we may assume that there are vertices $a_2, a_3\in A$,  $x_1\in X$ such that $a_1a_2x_1, a_1a_3x_1\in E(\dH)$.  

By Fact~\ref{Fact1} (ii), $a_1a_2x_1, a_1a_3x_1\in E(\dH)$ and $d_{\dH}(a_1)=d_{\dH}(a_2)=d_{\dH}(a_3)=2$ implies that $(a_1,a_2)$, $(a_1,a_3)$, and $(a_1,x_1)$ are bad. Then since $a_1a_2a_3\not\in E(\dH)$, we have $(a_2,a_3)$ is good. Thus there is a hyperedge $e_1$ such that $a_2, x_2\in e_1$ for some $x_2\in X$.
Since $N_{\dH}(a_1)=\{a_2,a_3,x_1\}$ contains only one vertex $x_1$ having degree at least 3, we have that for any $v\in V(\dH)\setminus\{a_1\}$, $(a_1,v)$ is bad. 
So for any pair of vertices $(v_1,v_2)\in V(\dH)\times V(\dH)$ which does not belong to $\{(a_2,x_1), (a_3,x_1)\}$, we have $(v_1,v_2)$ is good. 
However, if $e_1=a_2x_2x_3$ for some $x_3\in X\setminus\{x_1,x_2\}$, then Fact~\ref{Fact1} (ii) implies that $(x_2,x_3)$ is bad, a contradiction. 
If $e_1=a_2 x_2 a_4$ for some $a_4\in A\setminus\{a_1,a_2,a_3\}$, then Fact~\ref{Fact1} (ii) implies that $(a_2,a_4)$ is bad, a contradiction.  
Therefore, Case~(b) never occurs, and we are done.
\end{proof}

\noindent{\textbf{Completing the proof of $e(\dH)\ge n.$}}
For the remainder of the proof, by Lemma~\ref{Lem:no case a,b}, we may assume that no pair of hyperedges corresponds to Case~(a) or Case~(b). 
In order to estimate the number of hyperedges between $A$ and $X$ effectively, we now partition $A$ according to the types of hyperedges containing each vertex.
Let $A_1=\{a\in A: \mbox{both hyperedges containing } a \mbox{ are of Type 1}\}$, $A_2=\{a\in A: a \mbox{ has a neigbor in } A_1\}$, $A_3=\{a\in A\setminus (A_1\cup A_2): \mbox{exactly one hyperedge containing } a \mbox{ is of Type 1}\}$ and $A_4=\{a\in A: \mbox{both hyperedges containing } a \mbox{ are both of Type 2}\}$.

We first show that the sets $A_1,A_2,A_3,A_4$ partition $A$ and establish a lower bound on $|X|.$

\begin{clm}\label{clm: (A_1,A_2,A_3,A_4) partition A, X>25}
     The sets $A_1,A_2,A_3,A_4$ form a partition of $A$, and the number of hyperedges between $A$ and $X$ is at least $\frac{4}{3} |A|.$ 
     Moreover, $|X|\geq \frac{n}{4}\geq 24.$
\end{clm}
\begin{proof}
    By Lemma~\ref{Lem:no case a,b}, for any vertex $a_1 \in A_1$, there exist vertices $a_2, a_3 \in A_2$, $x_1, x_2 \in X$ such that $e_1 = a_1a_2x_1$ and $e_2 = a_1a_3x_2$. Let $e_3$ and $e_4$ be the other hyperedges containing $a_2$ and $a_3$, respectively. We first claim that both $e_3$ and $e_4$ are of Type 2. Indeed, since $a_1a_2x_1, a_1a_3x_2\in E(\dH)$, Fact~\ref{Fact1} (ii) implies that $(a_1, a_2)$, $(a_1, x_1)$, $(a_1, a_3)$, $(a_1, x_2)$, $(a_2, x_1)$ and $(a_3, x_2)$ are bad. Since $a_1a_2a_3\not\in E(\dH)$, $(a_2, a_3)$ is good and the core of the new Berge-$K_4$ in $\dH+a_2a_3$ contains $x_1$ and $x_2$, which forces $x_2\in e_3$ and $x_1\in e_4$. Furthermore, since $a_1a_2x_2\not\in E(\dH)$, the pair $(a_2, x_2)$ is good. The core of the new Berge-$K_4$ in $\dH+a_2x_2$ contains some vertex $x_3\in X\setminus \{x_1,x_2\}$, implying that $e_3$ is of Type 2. By a similar argument, $e_4$ is also of Type 2. 
    Thus $|A_2|=2|A_1|$ and the number of hyperedges between $A_1\cup A_2$ and $X$ is $2|A_1|+|A_2|=\frac{4}{3}(|A_1|+|A_2|)$. 

    Based on the analysis above, every vertex in $A_3$ has no neighbor in $A_2$. That is, for every $a_4\in A_3$, there is a vertex $a_5\in A_3$ such that $a_4a_5x_4\in E(\dH)$ for some $x_4\in X$, and there are two hyperedges containing $a_4$ and $a_5$ respectively which are of Type 2. Thus the number of hyperedges between $A_3$ and $X$ is $\frac{3}{2}|A_3|$.

    Therefore, $(A_1,A_2,A_3,A_4)$ is a partition of $A$, and the number of hyperedges between $A$ and $X$ is:
$$e(\dH[A,X])=\frac{4}{3}(|A_1|+|A_2|)+\frac{3}{2}|A_3|+2|A_4|\geq \frac{4}{3}|A|.$$
Furthermore, we have $|X|+|A|=n\geq e(\dH)=e(\dH[A,X])+e(\dH[X])\geq \frac{4}{3} |A|$, which implies that $|A|\leq 3|X|$. 
Thus $|X|\geq n/4\geq 96/4=24$, which completes the proof of the claim.
\end{proof}

We now complete the proof by considering whether or not there exists a bad pair in $A.$

If there is a pair $(a,a^\prime)\in A\times A$ is bad.
Let $X_1=(N_{\dH}(a)\cup N_{\dH}(a^\prime))\cap X$. Then $|X_1|\leq 8$. 
For any $v\in V(\dH)\setminus (\{a,a^\prime\}\cup X_1)$, $aa^\prime v \not\in E(\dH)$. Since $(a,a^\prime)$ is bad, at least one of $(a,v)$ and $(a^\prime,v)$ is good. So there are two vertices $x_i, x_j\in X_1$ and two hyperedges $e_v^1,e_v^2\in E(\dH)$ such that $v,x_i\in e_v^1$ and $v,x_j\in e_v^2$. By Lemma~\ref{Lem:no case a,b}, $|N_{\dH}(a,a^\prime)|\leq 1$. Hence there are at least $|V(\dH)\setminus (\{a,a^\prime\}\cup X_1)|\geq n-2-|X_1|=n-2-|X_1|$ disjoint hyperedges in $E(\dH-\{a,a^\prime\})$ containing at least one vertex from $X_1$. 
Therefore, we have $$\sum_{x\in X_1}d_{\dH}(x)\geq \sum_{x\in X_1}d_{\dH-\{a,a^\prime\}}(x)+|X_1|\geq n-2-|X_1|+|X_1|=n-2.$$ Hence, 
\begin{align*}
    e(\dH)  &\geq \left\lceil\frac{3(|X|-|X_1|)+n-2+2|A|}{3}\right\rceil
    = \left\lceil\frac{3n+|X|-(3|X_1|+2)}{3}\right\rceil\geq n,
\end{align*}
the last inequality holds since by Claim~\ref{clm: (A_1,A_2,A_3,A_4) partition A, X>25}, $|X|\geq 24\geq 3|X_1|$, and we are done.
\medskip

Now we may assume that every pair $(a,a^\prime)\in A\times A$ is good.
Fact~\ref{Fact1} (ii) implies that $(a_i,a_j)$ is bad if $a_ia_jx\in E(\dH)$ for some $x\in X$. Thus $A_1=A_2=A_3=\emptyset$, $A=A_4$.

Let $a_8\in A$ be a vertex, and let $a_8x_ax_b, a_8x_cx_d\in E(\dH)$ be the hyperedges incident to $a_8$ where $x_a,x_b,x_c,x_d\in X$. Set $\ell=|N_{\dH}(a_8)|=|\{x_a,x_b,x_c,x_d\}|$ and $A^\prime=A\setminus\{a_8\}$. Clearly we have $\ell\in \{3,4\}$.

For any $a\in A^\prime$, since $(a_8,a)$ is good, there are three hyperedges $e_a^1,e_a^2,e_a^3\in E(\dH)\setminus\{a_8x_ax_b, a_8x_cx_d\}$ such that $a,x_a^1\in e_a^1$, $a,x_a^2\in e_a^2$ and $x_a^1,x_a^2\in e_a^3$, where $x_a^1\in \{x_a,x_b\}$ and $x_a^2\in \{x_c,x_d\}$. Set $D=\cup_{a\in A'} \{x_a^1, x_a^2\}$, then $2\leq |D|\leq \ell\leq 4$.

If $|D|=2$, without loss of generality, we assume that $(x_a^1,x_a^2)=(x_a,x_c)$ for every $a\in A^\prime$. If there is a vertex $a$ such that $e_a^3$ is a hyperedge in $\dH[X]$, then $e_a^3=x_ax_cx$ for some $x\in X$. Thus $d_{\dH}(x_a)+d_{\dH}(x_c)\geq 2|A^\prime|+|\{a_8x_ax_b, a_8x_cx_d\}|+2|\{e_a^3\}
|=2(|A|-1)+2+2=2|A|+2$. If every hyperedge $e_a^3$ is between $A$ and $X$, then we may assume that $e_{a}^3=a_9x_ax_c$ for some $a, a_9\in A$. Since $(x_{a_9}^1,x_{a_9}^2)=(x_a,x_c)$ and $d_{\dH}(a_9)=2$, there is a vertex $a_{10}$ such that $e_{a_9}^3=a_{10}x_ax_c$. Hence, we have $d_{\dH}(x_a)+d_{\dH}(x_c)\geq 2|A^\prime|+|\{a_8x_ax_b, a_8x_cx_d\}|+|\{a_9x_ax_c, a_{10}x_ax_c\}|=2(|A|-1)+2+2=2|A|+2$.
Therefore $$e(\dH)\geq \left\lceil\frac{3(|X|-2)+2|A|+2+2|A|}{3}\right\rceil
    = \left\lceil\frac{3n+|A|-4}{3}\right\rceil\geq n,$$ where the last inequality follows from $|A|\geq 3$.

If $|D|=3$, then we may assume that $(x_a^1,x_a^2)=(x_a,x_c)$ or $(x_a,x_d)$ for every $a\in A^\prime$. That is, $D=\{x_a, x_c, x_d\}$. 
If $e_a^3=x_ax_cx_d$ for any $a\in A^\prime$, then $\sum_{x\in D}d_{\dH}(x)\geq 2|A^\prime|+2|\{a_8x_cx_d\}|+|\{a_8x_ax_b\}|+3|\{x_ax_cx_d\}|=2(|A|-1)+6=2|A|+4$. Otherwise, through an analysis similar to the case when $|D|=2$, we have $\sum_{x\in D}d_{\dH}(x)\geq 2(|A|-1)+2|\{a_8x_cx_d\}|+|\{a_8x_ax_b\}|+4=2|A|+5$. Thus $$e(\dH)\geq \left\lceil\frac{3(|X|-3)+2|A|+4+2|A|}{3}\right\rceil
    = \left\lceil\frac{3n+|A|-5}{3}\right\rceil\geq n,$$ where the last inequality follows from $|A|\geq 3$.

If $|D|=4$, then by the analysis similar to the cases when $|D|=2$ and $|D|=3$, we have $\sum_{x\in D}d_{\dH}(x)\geq 2(|A|-1)+2|\{a_8x_ax_c\}|+2|\{a_8x_bx_d\}|+4=2|A|+6$. Thus $$e(\dH)\geq \left\lceil\frac{3(|X|-4)+2|A|+6+2|A|}{3}\right\rceil
    = \left\lceil\frac{3n+|A|-6}{3}\right\rceil.$$
If $|A|\geq 4$, then we have $e(\dH)\geq \left\lceil\frac{3n+|A|-6}{3}\right\rceil=n$ and we are done . 

In what follows, we assume that $|A|=3$ and $|D|=\ell= 4$. Moreover, we claim that every vertex in $X$ has degree $3$. Otherwise we have $\sum_{x\in X}d_{\dH}(x)\geq 3|X|+1$, which implies $e(\dH)\geq \left\lceil\frac{\sum_{x\in X}d_{\dH}(x)+2|A|}{3}\right\rceil\geq \left\lceil\frac{3n-2}{3}\right\rceil= n$, and we are done. 
Let $A=\{a_8,a_9,a_{10}\}$, and let $x_7,x_8,x_9,x_{10}$ be $x_a,x_b,x_c,x_d$ respectively. 
Since $|D|=4,$ we may assume that $(x_{a_9}^1,x_{a_9}^2)=(x_7,x_9)$ and $(x_{a_{10}}^1,x_{a_{10}}^2)=(x_8,x_{10})$.  Since $x_7$ belongs to the three hyperedges $a_8x_7x_8, e_{a_9}^1,$ and $ e_{a_9}^3$, it can not be in any other hyperedges. Similarly, the same type of conclusion holds for $x_8,x_9$ and $x_{10}$.
By the assumption, $(a_9,a_{10})$ is good. Thus there are two vertices $x_{11}, x_{12}\in X\setminus\{x_7,x_8,x_9,x_{10}\}$ such that $\dH+a_9a_{10}$ contains a Berge-$K_4$ with core $\{a_9,a_{10},x_{11},x_{12}\}$ satisfies that $e_{a_9}^1=a_9x_7x_{11}$, $e_{a_9}^2=a_9x_9x_{12}$, $e_{a_{10}}^1=a_{10}x_8x_{12}$, $e_{a_{10}}^2=a_{10}x_{10}x_{11}$. (Here, if $e_{a_9}^1=a_9x_7x_{11}$, $e_{a_9}^2=a_9x_9x_{12}$, $e_{a_{10}}^1=a_{10}x_8x_{11}$, $e_{a_{10}}^2=a_{10}x_{10}x_{12}$, then by Fact~\ref{Fact1} (ii), $(x_7,x_8)$, $(x_7,x_{11})$ and $(x_8,x_{11})$ are bad, this contracts with $x_7x_8x_{11}\not\in E(\dH)$.) 
Moreover, there is a hyperedge $e_7$ such that $x_{11}, x_{12}\in e_7$. 

Consider $x_7x_8x_{11}\not\in E(\dH)$, since $(x_7,x_8)$ and $(x_7,x_{11})$ are bad, $(x_8, x_{11})$ is good. Assume that $e_7=x_{11}x_{12}v_1$ and $e_{a_9}^3=x_7x_9v_2$ and $e_{a_{10}}^3=x_8x_{10}v_3$. Then $N_{\dH}(x_8)=\{a_8, x_7,x_{10},x_{12},v_3\}$ and $N_{\dH}(x_{11})=\{a_9,a_{10},x_7,x_{10},x_{12},v_1\}$. However, $x_7\not\in N_{\dH}(x_{10})\cup N_{\dH}(x_{12})$, $x_{10}\not\in N_{\dH}(x_{12})\cup N_{\dH}(x_{7})$, $x_{12}\not\in N_{\dH}(x_{10})\cup N_{\dH}(x_{7})$, which implies that the core of a Berge-$K_4$ in $\dH+x_8x_{11}$ contains at most one vertex in $\{x_7,x_{10},x_{12}\}$. Thus $v_1=v_3$ and the Berge-$K_4$ in $\dH+x_8 x_{11}$ contains $e_7=x_{11}x_{12}v_1$ and $e_{a_{10}}^3=x_8x_{10}v_1$ which are mapped to $x_{11}v_1$ and $x_8v_1$ respectively. 
There is no hyperedge in $E(\dH)\setminus \{e_7,e_{a_{10}}^3\}$ can be mapped to $v_1x_{10}$ or $v_1x_{12}$, which implies that the core of Berge-$K_4$ in $\dH+x_8 x_{11}$ must be $\{x_8,x_{11},v_1,x_7\}$ and $v_2=v_1$.
That is, $e_{a_9}^3=x_7x_9v_1$, $e_{a_{10}}^3=x_8x_{10}v_1$ and $e_{7}=x_{11}x_{12}v_1$. 
Let $Z=\{a_8,a_9,a_{10},x_7,x_8,x_9,x_{10},x_{11}, x_{12}, v_1\}.$
Since $d_{\dH}(x)=3$ for any $x\in X$ and $d_{\dH}(a)=2$ for any $a\in A$, every vertex $v\in Z$ has $N_{\dH}(v)\subset Z$. That means for any $v\in V(\dH)\setminus Z$, $(v,x_7)$, $(v,x_8)$ are bad. However, $(x_7,x_8)$ is also bad and $x_7x_8v\not\in E(\dH)$, this contradicts the assumption that $\dH$ is Berge-$K_4$-saturated, completing the proof. \QED

\subsection{Proof of Theorem~\ref{mainthm}}\label{Sec: pf of main thm}
Let $n\geq 96$ be an integer and $\dH$ be an $n$-vertex $3$-uniform Berge-$K_4$-saturated hypergraph with the minimum number of edges. By Constructions \ref{Construction: odd n} and \ref{Construction: even n}, we have $e(\dH)\leq n$.
Form the argument in Subsection~\ref{Sec:lower bound n large}, we also have $e(\dH)\geq n$. Consequently, 
$$\sat_3(n,\text{Berge-}K_4)=n\ \text{holds\ for}\ n\ge 96.$$
For smaller $n$, a computer search yields 
$\sat_3(6,\text{Berge-}K_4)=5$ and $sat_3(n,\text{Berge-}K_4)=n$ for $n=5,7,8$ (see~\ref{Sec:Appendix A}).
Furthermore, by applying Lemma~\ref{Lemma: Add T}, we can generate numerous non-isomorphic families of extremal hypergraphs by adding $\dT$ on various pairs of vertices in the hypergraphs listed in Table~\ref{Table: Extremal hypergraphs}, which completes the proof of Theorem~\ref{mainthm}.\QED

\section{Algorithm for finding uniform Berge-\texorpdfstring{$K_\ell$}{} saturated hypergraphs}\label{Sec: Algorithm}

In this section, we introduce an algorithm for finding $k$-uniform Berge-$K_\ell$ saturated hypergraphs with given number of vertices and hyperedges, where $k$ and $\ell$ are non-negative integers. For expository convenience, we take $k=3$ and $\ell=4$. Let's recall that the incidence graph of a hypergraph is a bipartite graph with one part representing the hypergraph’s vertices and the other representing its hyperedges, where an edge connects a vertex to a hyperedge if the vertex belongs to that hyperedge.

The algorithm is straightforward. We first use \texttt{generate\_hypergraphs()} to generate all candidate hypergraphs via backtracking and get their incidence graphs, then use \texttt{select\_berge\_k4\_saturated\_p}\\ \texttt{arallel()} to select the ones corresponding to Berge-$K_4$-saturated hypergraphs in parallel, finally use \texttt{non\_isomorphic()} and \texttt{graph\_to\_hyper()} to obtain non-isomorphic ones in parallel and convert them back to hypergraphs.

\begin{breakablealgorithm}
   \caption{Searching $3$-Uniform Berge $K_4$ Saturated Hypergraphs}
   \label{alg:Searching}
   \begin{algorithmic}[1] 
   \REQUIRE $n\_vertices$, $n\_edges$, $uniform$, $n\_min\_degree$
   \ENSURE A list of non-isomorphic Berge $K_4$ saturated hypergraphs
   
   \LineComment{Generate all candidate hypergraphs via backtracking, and convert them into incidence graphs.}
\STATE candidates $\leftarrow$ \texttt{generate\_hypergraphs}($n\_vertices$, $n\_edges$, $uniform$, $n\_min\_degree$);

\LineComment{Parallel selection: for each $g \in$ candidates, \\
     \quad \quad \quad is\_berge\_k4\_saturated($g$) checks is\_berge\_k4\_free($g$) first,\\
     \quad \quad \quad then checks saturation.}
\STATE \LongState{berge\_k4\_saturated $\leftarrow$ \\
       \qquad \quad \quad \quad \texttt{select\_berge\_k4\_saturated\_parallel}(candidates);}

\STATE non\_isomorphic $\leftarrow$ \texttt{non\_isomorphic}(berge\_k4\_saturated);
\STATE result $\leftarrow$ \texttt{graph\_to\_hyper}(non\_isomorphic);
\RETURN result;
   \end{algorithmic}
   \end{breakablealgorithm}

   Considering the incidence graphs allows us to reduce the problem of deciding whether a hypergraph is Berge-$K_4$-saturated to the problem of finding maximum matchings in auxiliary graphs derived from its incidence graph. Moreover, it enables us to select non-isomorphic hypergraphs by using VF2++ algorithm \cite{AP2018}, which is an algorithm for graph isomorphism detection.

   We now turn to \texttt{select\_berge\_k4\_saturated\_parallel()}, which is a wrapper to execute \texttt{is\_berge}\\ \texttt{\_k4\_saturated()} in parallel. Generally speaking, \texttt{is\_berge\_k4\_saturated()} first uses \texttt{is\_berge\_k4\_f}\\ \texttt{ree()} to check whether a hypergraph is Berge-$K_4$-free, if so, it further tests whether the hypergraph is Berge-$K_4$-saturated. Let's introduce \texttt{is\_berge\_k4\_free()} first. For ease of illustration, we set the input to be a hypergraph (although it is actually an incidence graph).

   Let $\dH = (\dV, \dE)$ be the input hypergraph, we consider the $4$-element subsets of $\dV$. If there is a $4$-element subset that forms a Berge-$K_4$, then $\dH$ is not Berge-$K_4$-free; otherwise, it is. For each $4$-element subset $T$, we construct an auxiliary bipartite graph with $L:=\{p: p\subseteq T,|p|=2\}$ as its left part and $R:=\mathcal{E}$ as its right part. For $p\in L$ and $e\in R$, $p$ is adjacent to $e$ if and only if $p\subseteq e$. It is easy to see that if the matching number of the auxiliary bipartite is $6$, then $T$ forms a Berge-$K_4$ implies $\dH$ is not Berge-$K_4$-free; otherwise, $T$ can not form a Berge-$K_4$. This is how \texttt{is\_berge\_k4\_free()} works.

\begin{breakablealgorithm}
\caption{\texttt{is\_berge\_k4\_free}$(\dH)$}
\label{alg:is_berge_k4_free}
\begin{algorithmic}[1]
\REQUIRE A hypergraph $\dH = (\dV, \dE)$
\ENSURE \texttt{True} if $\dH$ is Berge $K_4$-free; \texttt{False} otherwise

\FOR{each 4-tuple $T = (v_1, v_2, v_3, v_4) \subseteq \dV$ with $|T| = 4$}
    \STATE Let $P_T \leftarrow$ all $\binom{4}{2} = 6$ unordered pairs from $T$;
    
    \STATE Construct a bipartite graph $\text{aux\_g} = (L \cup R, E)$ where:
    \begin{itemize}
        \item $L \leftarrow P_T$ (left part: the 6 pairs),
        \item $R \leftarrow \mathcal{E}$ (right part: all hyperedges of $H$),
        \item $(p, e) \in E$ iff $p \subseteq e$ (i.e., pair $p$ is contained in hyperedge $e$).
    \end{itemize}
    \LineComment{Compute a maximum matching $M$ in $\text{aux\_g}$.}
    \STATE $M$ $\leftarrow$ \texttt{nx.max\_weight\_matching}(aux\_g, maxcardinality=True); 
    
    \IF{$|M| = 6$}
        \RETURN \texttt{False} \Comment{Berge $K_4$ found}
    \ENDIF
\ENDFOR
\RETURN \texttt{True} \Comment{$H$ is Berge $K_4$-free}
\end{algorithmic}
\end{breakablealgorithm}

Next, we introduce \texttt{is\_berge\_k4\_saturated()}. In this algorithm, we first check wether $\dH$ is Berge-$K_4$-free. For each Berge-$K_4$-free hypergraph, we find out all its bad vertex pairs. If there are $3$ bad vertex pairs form a non-hyperedge, then by the definition of bad pair, adding this non-hyperedge does not create a Berge-$K_4$, thereby implying that the hypergraph is not Berge-$K_4$-saturated; otherwise, it is.

Recall that a vertex pair $(v_1,v_2)$ is defined as good if there exists two vertices $v_3,v_4$ and five hyperedges $e_{13},e_{14},e_{23},e_{24},e_{34}$ such that $\{v_i,v_j\}\subseteq e_{i,j}$ for all $(i,j)\neq (1,2)$. Furthermore, $v_3,v_4$ are common neighbors of $v_1$ and $v_2$ with each having degree at least $3$. Inspired by this, for a vertex pair $(u,v)$, we consider all $4$-element subsets $\{u,v,x,y\}$ where $x,y$ are two common neighbors of $u$ and $v$ with each having degree at least $3$. We construct auxiliary bipartite graphs similar to those in Algorithm~\ref{alg:is_berge_k4_free} and examine their matching numbers to determine wether $(u,v)$ is good.

\begin{breakablealgorithm}
\caption{\texttt{is\_berge\_k4\_saturated}$(\dH)$}
\label{alg:is_berge_k4_saturated}
\begin{algorithmic}[1]
\REQUIRE A hypergraph $\dH = (\dV, \mathcal{E})$
\ENSURE \texttt{True} if $\dH$ is Berge $K_4$-saturated; \texttt{False} otherwise

\IF{\texttt{is\_berge\_k4\_free}$(\dH) = \texttt{False}$}
    \RETURN \texttt{False}
\ENDIF\Comment{Must be Berge $K_4$-free to be saturated}

\STATE badPairs $\leftarrow \emptyset$;

\FOR{each unordered pair $(u, v)$ with $u \neq v$}
    \STATE \LongState{commonNeighbors $\leftarrow$\\ \qquad \quad \quad \quad$ \{ w \in \dV \setminus \{u,v\} \mid w \text{ adjacent to both }u \text{ and } v,\ \deg(w) \geq 3 \}$;}
    
    \IF{$|$commonNeighbors$| < 3$}
        \STATE add $(u, v)$ to badPairs;
        \CONTINUE;
    \ENDIF

    \STATE isGood $\leftarrow$ \texttt{False};

    \FOR{each unordered pair $(x, y) \subseteq$ commonNeighbors}
        \STATE Let $T \leftarrow (u, v, x, y)$;
        \STATE Let $P_T \leftarrow$ all $\binom{4}{2} - 1 = 5$ pairs from $T$ excluding $(u, v)$;
        
        \STATE Construct bipartite graph $\text{aux\_g} = (L \cup R, E)$ where:
        \begin{itemize}
            \item $L \gets P_T$,
            \item $R \gets \mathcal{E}$,
            \item $(p, e) \in E$ iff $p \subseteq e$.
        \end{itemize}
        
        \STATE Compute a maximum matching $M$ in $\text{aux\_g}$;
        
        \IF{$|M| = 5$}
            \STATE isGood $\leftarrow$ \texttt{True};
            \BREAK; \Comment{Found one witness: $(u,v)$ is good}
        \ENDIF
    \ENDFOR

    \IF{not isGood}
        \STATE add $(u, v)$ to badPairs;
    \ENDIF
\ENDFOR

\LineComment{Check saturation condition: every non-edge must create a Berge $K_4$.}
\FOR{each triple of distinct bad pairs $\{(a,b), (b,c), (a,c)\}$}
    \IF{$\{a,b,c\}$ induces a non-hyperedge in $H$}
        \RETURN \texttt{False};
    \ENDIF
\ENDFOR

\RETURN \texttt{True};
\end{algorithmic}
\end{breakablealgorithm}

\section{Concluding Remarks}\label{Sec: Conclusion}
In this paper, we have proved that $\sat_3(n,\text{Berge-}K_4)=n$ for $n=5,7,8$ and $n\geq 96$, while $\sat_3(6,\text{Berge-}K_4)=5$. 
The constant $96$ in our result could be reduced through a more detailed case analysis in the proof,  especially when handling the case there is a pair of vertices in $A$ is good. However, pushing this down to $n=9$ and thereby settling the problem for all $n$ appears to require substantially new ideas.
We conjecture that, in fact, $\sat_3(n,\text{Berge-}K_4)=n$ holds for every $n\geq 7$.

We have also developed an algorithmic framework for generating $k$-uniform Berge-$K_\ell$-saturated hypergraphs on a given number of vertices. This allowed us to determine the exact saturation numbers and to enumerate all extremal hypergraphs for $n=5,6,7,8$. The same approach can be applied to study saturation numbers for other uniformity parameters $k$ and clique sizes $\ell$.

Finally, we observed that many extremal hypergraphs for small $n$ can be extended to larger extremal hypergraphs by attaching the small configuration $\mathcal{T}$. This suggests a possible recursive structure: for sufficiently large $n$, every extremal hypergraph may be obtained by adding copies of $\mathcal{T}$ to an extremal hypergraph on fewer vertices.

\section*{Acknowledgement}
Tianying Xie was supported by National Key R and D Program of China  2023YFA1010201 and National Natural Science Foundation of China grants 12501474 and 12471336. 
\section*{Code Availability}
The computational code used to perform our algorithm in this paper and some useful tools are openly available on GitHub at \url{https://github.com/cyhcyh/Berge-K_4-saturation}.

\bibliographystyle{unsrt}

    \clearpage
\setcounter{subsection}{0} 
\renewcommand{\thesubsection}{Appendix \Alph{subsection}} 

    \subsection{Extremal hypergraphs for \texorpdfstring{$n = 5,6,7,8$}{}}\label{Sec:Appendix A}
    For small cases, we obtained $\sat_3(6,\text{Berge-}K_4)=5$ and $\sat_3(n,\text{Berge-}K_4)=n$ for $n=5,7,8$ through computer calculations. All extremal hypergraphs are listed in Table~\ref{Table: Extremal hypergraphs}. For hypergraphs with order $n$, we take the vertex set to be $\{0,1,2,...,n-1\}$. The last column of Table~\ref{Table: Extremal hypergraphs} denotes whether $\dT$ can be added to each hypergraph: the paires on which $\dT$ can be added are provided where applicable, and “No” is recorded otherwise.

    \begin{longtable}{
    L 
    C 
    H 
    P 
}
\caption{Extremal hypergraphs by order}\label{Table: Extremal hypergraphs} \\
\toprule
\textbf{Order} & \textbf{No.} & \textbf{Hyperedges} & \textbf{Can add $\dT$?} \\
\midrule
\endfirsthead

\toprule
\textbf{Order} & \textbf{No.} & \textbf{Hyperedges} & \textbf{Can add $\dT$?} \\
\midrule
\endhead

\bottomrule
\endfoot

\multirow{6}{*}{$5$} & \resetorder \stepcounter{ordercounter}\theordercounter & (0, 1, 2), (0, 1, 3), (0, 1, 4), (0, 2, 3), (0, 3, 4) & No \\
\cmidrule{2-4}
 & \stepcounter{ordercounter}\theordercounter & (0, 1, 2), (0, 1, 3), (0, 1, 4), (0, 2, 3), (1, 2, 4) & No \\
 \cmidrule{2-4}
 & \stepcounter{ordercounter}\theordercounter & (0, 1, 2), (0, 1, 3), (0, 1, 4), (0, 2, 3), (1, 2, 3) & No \\
 \cmidrule{2-4}
 & \stepcounter{ordercounter}\theordercounter & (0, 1, 2), (0, 1, 3), (0, 1, 4), (0, 2, 3), (2, 3, 4) & (0, 3), (1, 2) \\
 \cmidrule{2-4}
 & \stepcounter{ordercounter}\theordercounter & (0, 1, 2), (0, 1, 3), (0, 2, 3), (1, 2, 4), (1, 3, 4) & (2, 3) \\
 \cmidrule{2-4}
 & \stepcounter{ordercounter}\theordercounter & (0, 1, 2), (0, 1, 3), (0, 2, 4), (1, 3, 4), (2, 3, 4) & (0, 2), (0, 4) \\
\cmidrule{1-4}

\multirow{1}{*}{$6$} & \resetorder \stepcounter{ordercounter}\theordercounter & (0, 1, 2), (0, 1, 3), (0, 2, 4), (1, 3, 4), (2, 3, 4) & No \\
\cmidrule{1-4}

\multirow{20}{*}{$7$} & \resetorder \stepcounter{ordercounter}\theordercounter & (0, 1, 2), (0, 1, 3), (0, 1, 4), (0, 1, 5), (0, 2, 6), (2, 3, 4), (3, 5, 6) & No \\
\cmidrule{2-4}
& \stepcounter{ordercounter}\theordercounter & (0, 1, 2), (0, 1, 3), (0, 1, 4), (0, 2, 3), (0, 5, 6), (2, 3, 4), (2, 5, 6) & (0, 2) \\
\cmidrule{2-4}
& \stepcounter{ordercounter}\theordercounter & (0, 1, 2), (0, 1, 3), (0, 1, 4), (0, 2, 3), (0, 5, 6), (2, 3, 5), (2, 4, 6) & (0, 2) \\
\cmidrule{2-4}
& \stepcounter{ordercounter}\theordercounter & (0, 1, 2), (0, 1, 3), (0, 1, 4), (0, 2, 3), (1, 5, 6), (2, 3, 4), (2, 5, 6) & (1, 2) \\
\cmidrule{2-4}
& \stepcounter{ordercounter}\theordercounter & (0, 1, 2), (0, 1, 3), (0, 1, 4), (0, 2, 3), (1, 5, 6), (2, 3, 5), (2, 4, 6) & (1, 2) \\
\cmidrule{2-4}
& \stepcounter{ordercounter}\theordercounter & (0, 1, 2), (0, 1, 3), (0, 1, 4), (0, 2, 5), (0, 2, 6), (1, 3, 5), (1, 4, 6) & No \\
\cmidrule{2-4}
& \stepcounter{ordercounter}\theordercounter & (0, 1, 2), (0, 1, 3), (0, 1, 4), (0, 2, 5), (0, 2, 6), (1, 3, 5), (4, 5, 6) & No \\
\cmidrule{2-4}
& \stepcounter{ordercounter}\theordercounter & (0, 1, 2), (0, 1, 3), (0, 1, 4), (0, 2, 5), (0, 2, 6), (1, 5, 6), (2, 3, 4) & No \\
\cmidrule{2-4}
& \stepcounter{ordercounter}\theordercounter & (0, 1, 2), (0, 1, 3), (0, 1, 4), (0, 2, 5), (0, 2, 6), (1, 5, 6), (3, 4, 5) & No \\
\cmidrule{2-4}
& \stepcounter{ordercounter}\theordercounter & (0, 1, 2), (0, 1, 3), (0, 1, 4), (0, 2, 5), (0, 3, 6), (1, 2, 6), (1, 4, 5) & No \\
\multirow{34}{*}{$7$} & \stepcounter{ordercounter}\theordercounter & (0, 1, 2), (0, 1, 3), (0, 1, 4), (0, 2, 5), (0, 3, 6), (1, 4, 5), (2, 5, 6) & No \\
\cmidrule{2-4}
& \stepcounter{ordercounter}\theordercounter & (0, 1, 2), (0, 1, 3), (0, 1, 4), (0, 2, 5), (0, 3, 6), (2, 4, 5), (2, 4, 6) & No \\
\cmidrule{2-4}
& \stepcounter{ordercounter}\theordercounter & (0, 1, 2), (0, 1, 3), (0, 1, 4), (0, 2, 5), (0, 3, 6), (2, 4, 6), (3, 4, 5) & No \\
\cmidrule{2-4}
& \stepcounter{ordercounter}\theordercounter & (0, 1, 2), (0, 1, 3), (0, 1, 4), (0, 2, 5), (0, 5, 6), (1, 2, 6), (2, 3, 4) & No \\
\cmidrule{2-4}
& \stepcounter{ordercounter}\theordercounter & (0, 1, 2), (0, 1, 3), (0, 1, 4), (0, 2, 5), (0, 5, 6), (2, 3, 5), (2, 4, 6) & No \\
\cmidrule{2-4}
& \stepcounter{ordercounter}\theordercounter & (0, 1, 2), (0, 1, 3), (0, 1, 4), (0, 2, 5), (1, 2, 6), (2, 3, 4), (2, 5, 6) & No \\
\cmidrule{2-4}
& \stepcounter{ordercounter}\theordercounter & (0, 1, 2), (0, 1, 3), (0, 1, 4), (0, 2, 5), (1, 2, 6), (2, 3, 4), (3, 5, 6) & No \\
\cmidrule{2-4}
& \stepcounter{ordercounter}\theordercounter & (0, 1, 2), (0, 1, 3), (0, 1, 4), (0, 2, 5), (1, 5, 6), (2, 3, 4), (2, 5, 6) & (1, 2) \\
\cmidrule{2-4}
& \stepcounter{ordercounter}\theordercounter & (0, 1, 2), (0, 1, 3), (0, 1, 4), (0, 2, 5), (2, 3, 4), (2, 3, 6), (3, 5, 6) & No \\
\cmidrule{2-4}
& \stepcounter{ordercounter}\theordercounter & (0, 1, 2), (0, 1, 3), (0, 1, 4), (0, 2, 5), (2, 3, 4), (2, 3, 6), (4, 5, 6) & No \\
\cmidrule{2-4}
& \stepcounter{ordercounter}\theordercounter & (0, 1, 2), (0, 1, 3), (0, 1, 4), (0, 2, 5), (2, 3, 4), (3, 4, 6), (3, 5, 6) & No \\
\cmidrule{2-4}
& \stepcounter{ordercounter}\theordercounter & (0, 1, 2), (0, 1, 3), (0, 1, 4), (0, 2, 5), (2, 3, 5), (2, 3, 6), (2, 4, 6) & No \\
\cmidrule{2-4}
& \stepcounter{ordercounter}\theordercounter & (0, 1, 2), (0, 1, 3), (0, 1, 4), (0, 2, 5), (2, 3, 5), (2, 3, 6), (3, 4, 6) & No \\
\cmidrule{2-4}
& \stepcounter{ordercounter}\theordercounter & (0, 1, 2), (0, 1, 3), (0, 1, 4), (0, 2, 5), (2, 3, 5), (2, 4, 6), (3, 5, 6) & No \\
\cmidrule{2-4}
& \stepcounter{ordercounter}\theordercounter & (0, 1, 2), (0, 1, 3), (0, 1, 4), (0, 2, 5), (2, 3, 6), (3, 4, 5), (3, 4, 6) & No \\
\cmidrule{2-4}
& \stepcounter{ordercounter}\theordercounter & (0, 1, 2), (0, 1, 3), (0, 1, 4), (0, 2, 5), (2, 3, 6), (3, 4, 5), (4, 5, 6) & No \\
\cmidrule{2-4}
& \stepcounter{ordercounter}\theordercounter & (0, 1, 2), (0, 1, 3), (0, 2, 3), (1, 2, 4), (1, 3, 4), (2, 5, 6), (3, 5, 6) & (2, 3) \\
\multirow{34}{*}{$7$} & \stepcounter{ordercounter}\theordercounter & (0, 1, 2), (0, 1, 3), (0, 2, 3), (1, 2, 4), (1, 3, 5), (2, 5, 6), (3, 4, 6) & (2, 3) \\
\cmidrule{2-4}
& \stepcounter{ordercounter}\theordercounter & (0, 1, 2), (0, 1, 3), (0, 2, 4), (0, 3, 4), (0, 5, 6), (1, 2, 5), (1, 3, 6) & No \\
\cmidrule{2-4}
& \stepcounter{ordercounter}\theordercounter & (0, 1, 2), (0, 1, 3), (0, 2, 4), (0, 3, 4), (0, 5, 6), (1, 2, 5), (3, 4, 6) & No \\
\cmidrule{2-4}
& \stepcounter{ordercounter}\theordercounter & (0, 1, 2), (0, 1, 3), (0, 2, 4), (0, 3, 5), (0, 4, 6), (1, 3, 4), (1, 5, 6) & No \\
\cmidrule{2-4}
& \stepcounter{ordercounter}\theordercounter & (0, 1, 2), (0, 1, 3), (0, 2, 4), (0, 3, 5), (0, 4, 6), (1, 3, 4), (4, 5, 6) & No \\
\cmidrule{2-4}
& \stepcounter{ordercounter}\theordercounter & (0, 1, 2), (0, 1, 3), (0, 2, 4), (0, 3, 5), (0, 4, 6), (1, 3, 6), (4, 5, 6) & No \\
\cmidrule{2-4}
& \stepcounter{ordercounter}\theordercounter & (0, 1, 2), (0, 1, 3), (0, 2, 4), (0, 3, 5), (0, 4, 6), (1, 4, 6), (1, 5, 6) & No \\
\cmidrule{2-4}
& \stepcounter{ordercounter}\theordercounter & (0, 1, 2), (0, 1, 3), (0, 2, 4), (0, 3, 5), (0, 4, 6), (1, 5, 6), (3, 4, 5) & No \\
\cmidrule{2-4}
& \stepcounter{ordercounter}\theordercounter & (0, 1, 2), (0, 1, 3), (0, 2, 4), (0, 3, 5), (1, 2, 5), (1, 3, 6), (4, 5, 6) & No \\
\cmidrule{2-4}
& \stepcounter{ordercounter}\theordercounter & (0, 1, 2), (0, 1, 3), (0, 2, 4), (0, 3, 5), (1, 2, 6), (1, 4, 5), (4, 5, 6) & No \\
\cmidrule{2-4}
& \stepcounter{ordercounter}\theordercounter & (0, 1, 2), (0, 1, 3), (0, 2, 4), (0, 3, 5), (2, 3, 4), (2, 5, 6), (3, 4, 6) & (2, 3) \\
\cmidrule{2-4}
& \stepcounter{ordercounter}\theordercounter & (0, 1, 2), (0, 1, 3), (0, 2, 4), (0, 5, 6), (1, 2, 5), (1, 4, 6), (3, 5, 6) & No \\
\cmidrule{2-4}
& \stepcounter{ordercounter}\theordercounter & (0, 1, 2), (0, 1, 3), (0, 2, 4), (0, 5, 6), (1, 2, 5), (3, 4, 5), (3, 4, 6) & No \\
\cmidrule{2-4}
& \stepcounter{ordercounter}\theordercounter & (0, 1, 2), (0, 1, 3), (0, 2, 4), (0, 5, 6), (1, 3, 4), (1, 5, 6), (2, 3, 4) & (0, 1) \\
\cmidrule{2-4}
& \stepcounter{ordercounter}\theordercounter & (0, 1, 2), (0, 1, 3), (0, 2, 4), (0, 5, 6), (1, 3, 4), (2, 3, 4), (3, 5, 6) & (0, 3) \\
\cmidrule{2-4}
& \stepcounter{ordercounter}\theordercounter & (0, 1, 2), (0, 1, 3), (0, 2, 4), (0, 5, 6), (1, 3, 4), (2, 3, 5), (2, 3, 6) & (0, 3) \\
\cmidrule{2-4}
& \stepcounter{ordercounter}\theordercounter & (0, 1, 2), (0, 1, 3), (0, 2, 4), (0, 5, 6), (1, 3, 4), (2, 3, 5), (3, 5, 6) & (0, 3) \\
\multirow{8}{*}{$7$}& \stepcounter{ordercounter}\theordercounter & (0, 1, 2), (0, 1, 3), (0, 2, 4), (0, 5, 6), (1, 3, 4), (2, 4, 5), (3, 4, 6) & (0, 4) \\
\cmidrule{2-4}
& \stepcounter{ordercounter}\theordercounter & (0, 1, 2), (0, 1, 3), (0, 2, 4), (0, 5, 6), (1, 3, 4), (2, 4, 5), (4, 5, 6) & (0, 4) \\
\cmidrule{2-4}
& \stepcounter{ordercounter}\theordercounter & (0, 1, 2), (0, 1, 3), (0, 2, 4), (1, 3, 4), (2, 3, 5), (2, 3, 6), (4, 5, 6) & No \\
\cmidrule{2-4}
& \stepcounter{ordercounter}\theordercounter & (0, 1, 2), (0, 1, 3), (0, 2, 4), (1, 3, 5), (2, 3, 6), (2, 5, 6), (3, 4, 6) & No \\

\cmidrule{1-4}

\multirow{26}{*}{$8$} & \resetorder \stepcounter{ordercounter}\theordercounter & (0, 1, 2), (0, 1, 3), (0, 1, 4), (0, 2, 3), (0, 2, 5), (0, 6, 7), (3, 4, 5), (3, 6, 7) & (0, 3) \\
\cmidrule{2-4}
& \stepcounter{ordercounter}\theordercounter & (0, 1, 2), (0, 1, 3), (0, 1, 4), (0, 2, 3), (1, 2, 5), (1, 6, 7), (3, 4, 5), (3, 6, 7) & (1, 3) \\
\cmidrule{2-4}
& \stepcounter{ordercounter}\theordercounter & (0, 1, 2), (0, 1, 3), (0, 1, 4), (0, 2, 5), (0, 6, 7), (1, 3, 5), (2, 4, 5), (5, 6, 7) & (0, 5) \\
\cmidrule{2-4}
& \stepcounter{ordercounter}\theordercounter & (0, 1, 2), (0, 1, 3), (0, 1, 4), (0, 2, 5), (0, 6, 7), (2, 3, 6), (2, 3, 7), (3, 4, 5) & No \\
\cmidrule{2-4}
& \stepcounter{ordercounter}\theordercounter & (0, 1, 2), (0, 1, 3), (0, 2, 3), (0, 4, 5), (0, 6, 7), (1, 2, 4), (1, 3, 5), (1, 6, 7) & (0, 1) \\
\cmidrule{2-4}
& \stepcounter{ordercounter}\theordercounter & (0, 1, 2), (0, 1, 3), (0, 2, 4), (0, 3, 5), (0, 6, 7), (1, 2, 5), (1, 4, 6), (1, 4, 7) & No \\
\cmidrule{2-4}
& \stepcounter{ordercounter}\theordercounter & (0, 1, 2), (0, 1, 3), (0, 2, 4), (0, 3, 5), (0, 6, 7), (1, 2, 5), (3, 4, 5), (5, 6, 7) & (0, 5) \\
\cmidrule{2-4}
& \stepcounter{ordercounter}\theordercounter & (0, 1, 2), (0, 1, 3), (0, 2, 4), (0, 3, 5), (0, 6, 7), (1, 2, 6), (1, 3, 7), (1, 4, 5) & No \\
\cmidrule{2-4}
& \stepcounter{ordercounter}\theordercounter & (0, 1, 2), (0, 1, 3), (0, 2, 4), (0, 3, 5), (0, 6, 7), (1, 2, 6), (1, 4, 5), (1, 4, 7) & No \\
\cmidrule{2-4}
& \stepcounter{ordercounter}\theordercounter & (0, 1, 2), (0, 1, 3), (0, 2, 4), (0, 3, 5), (0, 6, 7), (1, 2, 6), (1, 4, 7), (5, 6, 7) & No \\
\cmidrule{2-4}
& \stepcounter{ordercounter}\theordercounter & (0, 1, 2), (0, 1, 3), (0, 2, 4), (0, 5, 6), (0, 5, 7), (1, 3, 4), (2, 3, 6), (3, 5, 7) & (0, 3) \\
\cmidrule{2-4}
& \stepcounter{ordercounter}\theordercounter & (0, 1, 2), (0, 1, 3), (0, 2, 4), (0, 5, 6), (0, 5, 7), (1, 3, 4), (2, 4, 5), (4, 6, 7) & (0, 4) \\
\cmidrule{2-4}
& \stepcounter{ordercounter}\theordercounter & (0, 1, 2), (0, 1, 3), (0, 2, 4), (0, 5, 6), (0, 5, 7), (1, 3, 4), (2, 4, 6), (4, 5, 7) & (0, 4) \\
\multirow{8}{*}{$8$} & \stepcounter{ordercounter}\theordercounter & (0, 1, 2), (0, 1, 3), (0, 2, 4), (0, 5, 6), (0, 5, 7), (1, 3, 5), (2, 3, 6), (3, 4, 7) & (0, 3) \\
\cmidrule{2-4}
& \stepcounter{ordercounter}\theordercounter & (0, 1, 2), (0, 1, 3), (0, 2, 4), (0, 5, 6), (0, 5, 7), (1, 3, 6), (2, 3, 7), (3, 4, 5) & (0, 3) \\
\cmidrule{2-4}
& \stepcounter{ordercounter}\theordercounter & (0, 1, 2), (0, 1, 3), (0, 4, 5), (0, 4, 6), (1, 4, 7), (1, 5, 6), (2, 4, 7), (3, 5, 7) & No \\
\cmidrule{2-4}
& \stepcounter{ordercounter}\theordercounter & (0, 1, 2), (0, 1, 3), (0, 4, 5), (0, 4, 6), (1, 5, 7), (2, 3, 5), (2, 4, 7), (3, 5, 6) & No \\
\bottomrule
\end{longtable}

\end{document}